\documentclass[a4paper, 11pt, american]{amsart}

\usepackage[colorlinks,pdfpagelabels,pdfstartview = FitH,bookmarksopen
= true,bookmarksnumbered = true,linkcolor = blue,plainpages =
false,hypertexnames = false,citecolor = red,pagebackref=false]{hyperref}
\usepackage{times,latexsym,amssymb}
\usepackage{amsmath,amsthm,bm}
\usepackage{graphics, epsfig}
\usepackage{color}
\usepackage{appendix}
\usepackage[margin=1.3in]{geometry}
\usepackage{mathrsfs}

\usepackage{amsbsy}
\usepackage{amstext}
\usepackage{amssymb}
\usepackage{esint}
\usepackage{stmaryrd}
\allowdisplaybreaks

\renewcommand{\d}{\mathrm{d}}
\newcommand{\dx}{\mathrm{d}x}
\newcommand{\dy}{\mathrm{d}y}

\newcommand{\dt}{\mathrm{d}t}
\newcommand{\ds}{\mathrm{d}s}
\renewcommand{\rho}{\varrho}

\let\TeXchi\chi
\newbox\chibox
\setbox0 \hbox{\mathsurround0pt $\TeXchi$}
\setbox\chibox \hbox{\raise\dp0 \box 0 }
\def\chi{\copy\chibox}

\def\Xint#1{\mathchoice
    {\XXint\displaystyle\textstyle{#1}}%
    {\XXint\textstyle\scriptstyle{#1}}%
    {\XXint\scriptstyle\scriptscriptstyle{#1}}%
    {\XXint\scriptscriptstyle\scriptscriptstyle{#1}}%
    \!\int}
\def\XXint#1#2#3{\setbox0=\hbox{$#1{#2#3}{\int}$}
    \vcenter{\hbox{$#2#3$}}\kern-0.5\wd0}
\def\bint{\Xint-}
\def\dashint{\Xint{\raise4pt\hbox to7pt{\hrulefill}}}

\def\XXiint#1#2#3{\setbox0=\hbox{$#1{#2#3}{\iint}$}
    \vcenter{\hbox{$#2#3$}}\kern-0.5\wd0}

\author[N. Liao]{Naian Liao}
\address{Naian Liao,
Fachbereich Mathematik, Paris-Lodron-Universit\"at Salzburg,
Hellbrunner Str. 34, 5020 Salzburg, Austria}
\email{naian.liao@plus.ac.at}
\newtheorem{proposition}{Proposition}[section]
\newtheorem{theorem}{Theorem}[section]
\newtheorem{lemma}{Lemma}[section]

\newtheorem{remark}{Remark}[section]

\numberwithin{equation}{section}
\numberwithin{theorem}{section}
\numberwithin{proposition}{section}
\numberwithin{lemma}{section}
\numberwithin{remark}{section}
\newcommand{\noi}{\noindent}

\newcommand{\dsty}{\displaystyle}

\newcommand{\al}{\alpha}

\newcommand{\be}{\beta}

\newcommand{\gm}{\gamma}
\newcommand{\dl}{\delta}

\newcommand{\lm}{\lambda}

\newcommand{\varep}{\varepsilon}

\newcommand{\vp}{\varphi}
\newcommand{\sig}{\sigma}

\newcommand{\om}{\omega}

\newcommand{\z}{\zeta}

\newcommand{\nn}{\mathbb{N}}

\newcommand{\rr}{\mathbb{R}}
\newcommand{\rn}{\rr^N}

\newcommand{\supp}{\operatorname{supp}}
\newcommand{\essup}{\operatornamewithlimits{ess\,sup}}
\newcommand{\essinf}{\operatornamewithlimits{ess\,inf}}
\newcommand{\essosc}{\operatornamewithlimits{ess\,osc}}

\newcommand{\loc}{\operatorname{loc}}

\newcommand{\pl}{\partial}

\newcommand{\intl}{\int\limits}

\def\Xint#1{\mathchoice
    {\XXint\displaystyle\textstyle{#1}}%
    {\XXint\textstyle\scriptstyle{#1}}%
    {\XXint\scriptstyle\scriptscriptstyle{#1}}%
    {\XXint\scriptscriptstyle\scriptscriptstyle{#1}}%
    \!\int}
\def\XXint#1#2#3{\setbox0=\hbox{$#1{#2#3}{\int}$}
    \vcenter{\hbox{$#2#3$}}\kern-0.5\wd0}
\def\bint{\Xint-}
\def\dashint{\Xint{\raise4pt\hbox to7pt{\hrulefill}}}
\def\dashiint{\bint\kern-0.15cm\bint}

\newcommand{\ovl}[3]{\int_{#1}^{#2}\kern-#3pt\raise4pt\hbox to7pt{\hrulefill}\ }

\newcommand{\ovll}[3]{\intl_{#1}^{#2}\kern-#3pt\raise4pt\hbox to7pt{\hrulefill}\ }

\newcommand{\tvl}[2]{\iint_{#1}\kern-#2pt\raise4pt\hbox to7pt{\hrulefill}\ }

\newcommand{\bye}{\end{document}}

\begin{document}
\title[Continuity of solutions to nonlocal parabolic equations]{On the modulus of continuity of solutions to\\  nonlocal parabolic equations}
\maketitle
\begin{abstract}
A general modulus of continuity is quantified for locally bounded, local, weak solutions to nonlocal parabolic equations, 
under a minimal tail condition. H\"older modulus of continuity is then deduced under a slightly stronger tail condition.
These regularity estimates are demonstrated under the framework of nonlocal $p$-Laplacian with measurable kernels.

\end{abstract}
\section{Introduction}

This note aims to extend the regularity theory of \cite{Liao-cvpd-24} and establish, under a minimal condition on the tail, the local continuity of weak solutions to
a class of parabolic equations involving a nonlocal operator of $p$-Laplacian type:
\begin{equation}\label{Eq:1:1}
\pl_t u + \mathscr{L} u=0\quad\text{weakly in}\> E_T:=E\times(0,T],
\end{equation}
for some open set $E\subset\rn$ and some $T>0$. 
The operator $\mathscr{L}$ is defined in \eqref{Eq:1:2} -- \eqref{Eq:K} with positive constants $C_o$, $C_1$ and $s\in(0,1)$. 
Our main regularity result reads as follows.

\begin{theorem}[general modulus of continuity]\label{Thm:1:1}
Let $u$ be a locally bounded, local, weak solution to \eqref{Eq:1:1} in $E_T$ satisfying \eqref{Eq:1:2} -- \eqref{Eq:K} with $p>1$.
	Then $u$ is locally continuous in $E_T$. More precisely,
	there exist constants $\boldsymbol\gm>1$ and $\be,\,\sig\in(0,1)$ depending on the data $\{s, p, N, C_o, C_1\}$, such that for any $0<r<R<\widetilde{R}$, there holds
	\begin{equation*}
	\essosc_{(x_o,t_o)+Q_{\sig r}(\boldsymbol\om^{2-p})}u \le 2 \boldsymbol\om \Big(\frac{r}{R}\Big)^{\be}  + \boldsymbol \gm  \int^{t_o}_{t_o-\boldsymbol\om^{2-p}(rR)^{sp/2}}    \int_{\rn\setminus K_{\widetilde{R}}(x_o)} \frac{|u(x,t)|^{p-1}}{|x-x_o|^{N+sp}}\,\dx \dt,
	\end{equation*}
provided the  cylinders  $(x_o,t_o)+Q_{R}(\boldsymbol\om^{2-p})\subset (x_o,t_o)+Q_{\widetilde{R}}$  
are included in $E_T$, 
where
$$\boldsymbol\om=2\essup_{(x_o,t_o)+Q_{\widetilde{R}}}|u| +\int^{t_o}_{t_o - \widetilde{R}^{sp}}    \int_{\rn\setminus K_{\widetilde{R}}(x_o)} \frac{|u(x,t)|^{p-1}}{|x-x_o|^{N+sp}}\,\dx \dt.$$
\end{theorem}

Based on Theorem~\ref{Thm:1:1}, one  obtains various moduli of continuity by prescribing different conditions on the solution's long-range behavior. A particularly interesting case concerns the H\"older regularity.
\begin{theorem}[H\"older modulus of continuity]\label{Thm:1:2}
Let $u$ be a locally bounded, local, weak solution to \eqref{Eq:1:1} in $E_T$ satisfying \eqref{Eq:1:2} -- \eqref{Eq:K} with $p>1$.
Assume additionally that for some $\varep>0$,
	\begin{equation*}
	\int_{\rn}\frac{|u(x,\cdot)|^{p-1}}{1+|x|^{N+sp}}\,\dx\in L^{1+\varep}_{\loc}(0,T].
	\end{equation*}
	Then $u$ is locally H\"older continuous in $E_T$. More precisely,
	there exist a constant $\boldsymbol\gm>1$ depending the data $\{s, p, N, C_o, C_1\}$ and a constant $\be\in(0,1)$ depending on the data and $\varep$, such that for any $0<r<R<\widetilde{R}$, there holds
	\begin{equation*}
	\essosc_{(x_o,t_o)+Q_r(\boldsymbol\om^{2-p})}u \le  \boldsymbol\gm\boldsymbol\om \Big(\frac{r}{R}\Big)^{\be},
	\end{equation*}
provided the  cylinders  $(x_o,t_o)+Q_{R}(\boldsymbol\om^{2-p})\subset (x_o,t_o)+Q_{\widetilde{R}}$ 
are included in $E_T$, 
where
$$\boldsymbol\om=2\essup_{(x_o,t_o)+Q_{\widetilde{R}}}|u| +\bigg(\bint^{t_o}_{t_o- \widetilde{R}^{sp}}  \Big( \widetilde{R}^{sp} \int_{\rn\setminus K_{\widetilde{R}}(x_o)} \frac{|u(x,t)|^{p-1}}{|x-x_o|^{N+sp}}\,\dx\Big)^{ 1+\varep} \dt\bigg)^{\frac1{1+\varep}}.$$
\end{theorem}

Throughout this note,  the parameters  $\{s, p, N, C_o, C_1\}$ are termed the  {\it data}, and we use $\boldsymbol\gm$ as a generic positive constant in various estimates that can be determined by the data only.
\subsection{Novelty and Significance}
Under a natural functional analytic set-up, weak (variational) solutions to the nonlocal elliptic equation $\mathscr{L} u=0$ with \eqref{Eq:1:2} -- \eqref{Eq:K} are known to be locally H\"older continuous, cf.~\cite{Cozzi, DKP-1, Kass-09}.
 However, nonlocal parabolic problems present a unique feature. That is, while the local behavior of solutions  still adheres to the continuity code for diffusion equations, the long-range behavior of a solution might break its local H\"older continuity.
 The regularity estimate presented in Theorem~\ref{Thm:1:1} confines the local behavior of solutions in this scenario. 
 Whereas Theorem~\ref{Thm:1:2} provides a sharp tail condition for H\"older regularity to hold.

H\"older estimates of weak solutions to nonlocal parabolic equations with measurable kernels have drawn considerable attention in the last decade, cf.~\cite{Caff-Vass-11, Kass-13}. Most recently, efforts were made   in \cite{Adi, DZZ, Liao-cvpd-24} to extend the parabolic theory to a nonlinear setting like \eqref{Eq:1:1} -- \eqref{Eq:K} and to manage the long-range behavior of solutions.  Previously, it was common to require a stronger tail condition than the one in Theorem~\ref{Thm:1:2}, namely, not only $L_{\loc}^{1+\varep}$ but $L_{\loc}^\infty$ in the time variable is required.   
Things changed with \cite{Kass-23}: The authors established, in the linear case, the first H\"older estimate under a tail condition with $L_{\loc}^{1+\varep}$ in the time variable. They also provided an example that tests its sharpness.
Very recently, a result similar to Theorem~\ref{Thm:1:2}  has been obtained in \cite{Byun-ampa}. 
 However, all previous works circumvented a more general result like Theorem~\ref{Thm:1:1} under a tail condition with the minimal $L_{\loc}^{1}$ integrability in time.

Although Theorem~\ref{Thm:1:1} is novel for linear equations, we choose to present the theory under the more general framework of $p$-Laplacian type operator in order to emphasize some universal principles in regularity estimates. The method of intrinsic scaling is combined with a fine control of the long-range behavior of solutions. 
This approach evolves out of the one developed in \cite{Liao-cvpd-24}. While all previous techniques can readily deal with a particular case -- H\"older modulus of continuity, our techniques  are flexible enough to obtain a general modulus of continuity. 

\subsection{Definitions and Notation}\label{S:notion}
The nonlocal operator $\mathscr{L}$ is defined by
\begin{equation}\label{Eq:1:2}
\mathscr{L}u(x,t)={\rm P.V.}\int_{\rn} K(x,y,t)\big|u(x,t) - u(y,t)\big|^{p-2} \big(u(x,t) - u(y,t)\big)\,\dy,
\end{equation}
for some $p>1$,
whereas 
 the kernel $K:\rn\times\rn\times(0,T]\to [0,\infty)$ is measurable and satisfies the following condition uniformly in $t$:
\begin{equation}\label{Eq:K}
\frac{C_o}{|x-y|^{N+sp}}\le K(x,y,t)\equiv K(y,x,t)\le \frac{C_1}{|x-y|^{N+sp}}\quad\text{a.e.}\> x,\,y\in\rn,
\end{equation}
for some positive $C_o$, $C_1$ and $s\in(0,1)$. 

A measurable function $u:\,\rn\times(0,T]\to\rr$ satisfying
\begin{equation*} 
	u\in C_{\loc}\big(0,T;L^2_{\loc}(E)\big)\cap L^p_{\loc}\big(0,T; W^{s,p}_{\loc}(E)\big)
\end{equation*}
is a local, weak sub(super)-solution to \eqref{Eq:1:1} -- \eqref{Eq:K}, if for every compact set $\mathcal{K}\subset E$ and every sub-interval
$[t_1,t_2]\subset (0,T]$, we have
\begin{equation}\label{Eq:global-int}
\int_{t_1}^{t_2}\int_{\rn}\frac{|u(x,t)|^{p-1}}{1+|x|^{N+sp}}\,\dx<\infty
\end{equation}
and
\begin{equation*} 
\begin{aligned}
	\int_{\mathcal{K}} & u\vp \,\dx\bigg|_{t_1}^{t_2}
	-\int_{t_1}^{t_2}\int_{\mathcal{K}}  u\pl_t\vp\, \dx\dt
	+\int_{t_1}^{t_2} \mathscr{E}\big(u(\cdot, t), \vp(\cdot, t)\big)\,\dt
	\le(\ge)0
\end{aligned}
\end{equation*}
where
\[
	 \mathscr{E}:=\int_{\rn}\int_{\rn}K(x,y,t)\big|u(x,t) - u(y,t)\big|^{p-2} \big(u(x,t) - u(y,t)\big)\big(\vp(x,t) - \vp(y,t)\big)\,\dy\dx
\]
for all non-negative testing functions
\begin{equation*} 
\vp \in W^{1,2}_{\loc}\big(0,T;L^2(\mathcal{K})\big)\cap L^p_{\loc}\big(0,T;W_o^{s,p}(\mathcal{K})%
\big).
\end{equation*}

A function $u$ that is both a local weak sub-solution and a local weak super-solution
to \eqref{Eq:1:1} -- \eqref{Eq:K} is a local weak solution.

The only difference from our previous notion of solution in \cite{Liao-cvpd-24} is the global condition \eqref{Eq:global-int}. Indeed, now we only need $L^1_{\loc}$ not $L^{\infty}_{\loc}$ in the time variable. A naturally connected concept is the tail term
\begin{equation}\label{Eq:tail}
{\rm Tail}\big[u; Q(R,S)\big]:= \int_{t_o-S}^{t_o} \int_{\rn\setminus K_R(x_o)}\frac{|u(x,t)|^{p-1}}{|x-x_o|^{N+sp}}\,\dx \dt.
\end{equation}
Here, and in what follows, 
we will use  the symbols 
\begin{equation*}
\left\{
\begin{aligned}
(x_o,t_o)+Q(R,S)&:=K_R(x_o)\times (t_o-S,t_o),\\[5pt]
(x_o,t_o)+Q_\rho(\theta)&:=K_{\rho}(x_o)\times(t_o-\theta\rho^{sp},t_o),
\end{aligned}\right.
\end{equation*} 
to denote (backward) cylinders, where $K_\rho(x_o)$ denotes the ball of radius $\rho$ and center $x_o$ in $\rn$.
The vertex $(x_o,t_o)$ has been omitted from the cylinder in \eqref{Eq:tail} for simplicity. 
If $\theta=1$, it will also be omitted. When the context is unambiguous, we will apply these conventions.

\

\noi{\bf Acknowledgement.} This work was supported by the FWF-project P36272-N ``On the Stefan type problems".

\section{Energy Estimates}\label{S:energy}
%
%
%

The energy estimates for truncated functions parallel those in \cite[Proposition~2.1]{Liao-cvpd-24}. A notable difference lies in the use of a time-dependent truncation level $k(t)$. This idea is taken from the recent work \cite{Kass-23}.

\begin{proposition}\label{Prop:2:1}
	Let $u$ be a  local weak sub(super)-solution to \eqref{Eq:1:1} -- \eqref{Eq:K} in $E_T$, and
	let $k(\cdot)$ be  absolutely continuous  in $(0,T)$.
	There exists a constant $\boldsymbol \gm (C_o,C_1,p)>0$, such that
 	for all cylinders $Q(R,S) \subset E_T$,
 	and every non-negative, piecewise smooth cutoff function
 	$\z(\cdot,t)$ compactly supported in $ K_{R} $ for all  $t\in (t_o-S,t_o)$,  there holds
\begin{align*}
	&\int_{t_o-S}^{t_o}\int_{K_R}\int_{K_R}\min\big\{\z^p(x,t),\z^p(y,t)\big\} \frac{|w_\pm(x,t) - w_\pm(y,t)|^p}{|x-y|^{N+sp}}\,\dx\dy\dt\\
	&\qquad+\iint_{Q(R,S)} \z^p w_{\pm}(x,t)\,\dx\dt \bigg(\int_{ K_R}  \frac{w^{p-1}_\mp(y,t)}{|x-y|^{N+sp}}\,\dy\bigg)
	 +\int_{K_R}\z^p w^2_{\pm}(x,t)\,\dx\bigg|_{t_o-S}^{t_o}\\
	&\quad\le
	\boldsymbol\gm\int_{t_o-S}^{t_o}\int_{K_R}\int_{K_R}\max\big\{w^p_{\pm}(x,t), w^p_{\pm}(y,t)\big\} \frac{|\z(x,t) - \z(y,t)|^p}{|x-y|^{N+sp}}\,\dx\dy\dt\\
	&\qquad+\boldsymbol\gm\int_{t_o-S}^{t_o}\int_{K_R}\int_{\rn\setminus K_R} \z^p w_{\pm}(x,t)\frac{ w_{\pm}^{p-1}(y,t)}{|x-y|^{N+sp}}\,\dy\dx\dt 
	\\
	&\qquad \mp 2\iint_{Q(R,S)}  k'(t) \z^pw_{\pm}(x,t)\,\dx\dt + \iint_{Q(R,S)} |\pl_t\z^p|w_{\pm}^2(x,t)\,\dx\dt.  
\end{align*}
Here, we have denoted $w(x,t)=u(x,t)-k(t)$ for simplicity.
\end{proposition}
\begin{proof}
Take the case of sub-solution for instance. Using $\vp=w_+\z^p$ as a testing function in the weak formulation, the integral resulting from the fractional diffusion term is treated as in \cite[Proposition~2.1]{Liao-cvpd-24}. Regarding the part with the time derivative, formally we can write
\[
 \pl_t u \cdot (u-k)_+\z^p =\tfrac12  \z^p\pl_t (u-k)_+^2  + \z^p k^{\prime}(u-k)_+ .
\]
Integrating this equality in $Q(R,S)$ readily yields the conclusion. A rigorous treatment of the time derivative can be adapted from 
\cite[Appendix~B]{Liao-cvpd-24}.
\end{proof}


\section{Preliminary Tools} \label{S:prelim}



The results of this section parallel those in \cite[Section~3]{Liao-cvpd-24}. The most notable difference is the tail alternative that appears in the statements to follow. A large portion of the proofs can be reproduced, and hence
we only sketch them while highlighting the main modifications.

Let us introduce the reference cylinder
 $\mathcal{Q}:=K_R(x_o)\times(T_1,T_2]\subset E_T$.
Suppose the quantities $\boldsymbol\mu^{\pm}$ and $\boldsymbol\om$ satisfy
\begin{equation*}
	\boldsymbol\mu^+\ge\essup_{\mathcal{Q}}u,
	\quad 
	\boldsymbol\mu^-\le\essinf_{\mathcal{Q}} u,
	\quad
	\boldsymbol\om\ge\boldsymbol\mu^+-\boldsymbol\mu^-.
\end{equation*}

The first two results will employ an iteration \`a la DeGiorgi. The tail alternative appears under weaker conditions than \cite[Lemmas~3.1 \& 3.2]{Liao-cvpd-24}, thanks to the time-dependent truncation.
For ease of notation, the vertex $(x_o,t_o)$ is omitted from $Q_\rho(\theta)$. 
\begin{lemma}\label{Lm:DG:1}
 Let $u$ be a locally bounded, local weak sub(super)-solution to \eqref{Eq:1:1} -- \eqref{Eq:K} in $E_T$.
 For some $ \dl,\,\xi\in(0,1)$ and $\rho\in(0,\frac12R)$, set $\theta=\dl(\xi\boldsymbol\om)^{2-p}$, and assume $ Q_\varrho(\theta) \subset \mathcal{Q}$.
There exist  $\widetilde{\boldsymbol\gm}>1$ depending only on 
 the  data $\{s, p, N, C_o, C_1\}$ and   $\nu\in(0,1)$ depending on the data and $\dl$, such that if
\begin{equation*}
	\big|\big\{
	\pm\big(\boldsymbol \mu^{\pm}-u\big)\le  \xi\boldsymbol\om\big\}\cap  Q_{\varrho}(\theta)\big|
	\le
	\nu|Q_{\varrho}(\theta)|,
\end{equation*}
then either
\begin{equation*}
\widetilde{\boldsymbol\gm} 
{\rm Tail}\big[\big(u - \boldsymbol \mu^{\pm}\big)_\pm; \mathcal{Q}\big]
>\xi\boldsymbol\om,
\end{equation*}
or
\begin{equation*}
	\pm\big(\boldsymbol\mu^{\pm}-u\big)\ge \tfrac14\xi\boldsymbol\om
	\quad
	\mbox{a.e.~in $ Q_{\frac{1}2\varrho}(\theta)$.}
\end{equation*}
Moreover, we have the dependence $\nu\approx  \dl^q$ for some $ q>1$ depending on $p$ and $N$.
\end{lemma}

\begin{proof}
Let us assume $(x_o,t_o)=(0,0)$ and show the case of super-solution with $\boldsymbol\mu^-=0$. Like in \cite[Lemma~3.1]{Liao-cvpd-24} we 
 introduce $k_n$, $\rho_n$, $\tilde{\rho}_n$, $\hat{\rho}_n$, $\bar{\rho}_n$, $K_n$, $\widetilde{K}_n$, $\widehat{K}_n$, $\overline{K}_n$, $Q_n$, $\widetilde{Q}_n$, $\widehat{Q}_n$, $\overline{Q}_n$ and a cutoff function $\z$ in $Q_n$.
The energy estimate of Proposition~\ref{Prop:2:1} is used in $Q_n$ with $\z$ and
with 
$$
w_-(x,t):=\big(u(x,t)+\ell(t)-k_n\big)_-, \quad \ell(t):=\widetilde{\boldsymbol\gm}\int^t_{T_1}\int_{\rn\setminus K_R}\frac{ u_{-}^{p-1}(y,\tau)}{|y|^{N+sp}}\,\dy\d\tau,
$$
where $\widetilde{\boldsymbol\gm}>1$ is to be determined.
As a result, we have
\begin{align}\nonumber
	\essup_{-\theta\tilde\rho_n^{sp}<t<0}&\int_{\widetilde{K}_n} w^2_{-}(x,t)\,\dx+ \int_{-\theta\tilde\rho_n^{sp}}^{0}\int_{\widetilde{K}_n}\int_{\widetilde{K}_n}  \frac{|w_-(x,t) - w_-(y,t)|^p}{|x-y|^{N+sp}}\,\dx\dy\dt\\\nonumber
	&\quad\le
	\boldsymbol\gm\int_{-\theta\rho_n^{sp}}^{0}\int_{K_n}\int_{K_n}\max\big\{w^p_{-}(x,t), w^p_{-}(y,t)\big\} \frac{|\z(x,t) - \z(y,t)|^p}{|x-y|^{N+sp}}\,\dx\dy\dt\\\nonumber
	&\qquad+\boldsymbol\gm\iint_{Q_n} \z^p w_{-}(x,t)\,\dx\dt \bigg(\essup_{\substack{x\in \widehat{K}_n}} \int_{\rn\setminus K_n}\frac{ w_{-}^{p-1}(y,t)}{|x-y|^{N+sp}}\,\dy\bigg)\\\label{Eq:energy-DG}
	&\qquad-2\iint_{Q_n}  \ell^{\prime}(t) \z^pw_{-}(x,t)\,\dx\dt + \iint_{Q_n} |\pl_t\z^p|w_{-}^2(x,t)\,\dx\dt.
\end{align}

The four terms on the right-hand side of \eqref{Eq:energy-DG} are treated as follows.
The first term is standard. Indeed, we estimate
\begin{align*}
\int_{-\theta\rho_n^{sp}}^{0}&\int_{K_n}\int_{K_n}\max\big\{w^p_{-}(x,t), w^p_{-}(y,t)\big\} \frac{|\z(x,t) - \z(y,t)|^p}{|x-y|^{N+sp}}\,\dx\dy\dt\\
&\le 2^{pn+1} \frac{(\xi\boldsymbol\om)^p}{\rho^{p}}\int_{-\theta\rho_n^{sp}}^{0}\int_{K_n}\int_{K_n} \frac{ \chi_{\{u(x,t)+\ell(t)<k_n\}} }{|x-y|^{N+(s-1)p}}\,\dx\dy\dt\\
&\le \boldsymbol \gm 2^{pn} \frac{(\xi\boldsymbol\om)^p}{\rho^{sp}} |A_n|,
\end{align*}
where we have defined $A_n:=\{u(x,t)+\ell(t)<k_n\}\cap Q_n$.
The last term is also standard, namely,
\[
\iint_{Q_n} |\pl_t\z^p |(u+\ell(t)-k_n)_{-}^2\,\dx\dt\le\frac{2^{spn}}{\theta\rho^{sp}} (\xi\boldsymbol\om)^2 |A_n|.
\]
 
The second term and the third, negative term need to be packed.  To this end, observe that when $|y|\ge \rho_n$ and $|x|\le \hat\rho_n$, there holds
\[
\frac{|y-x|}{|y|}\ge1-\frac{\hat\rho_n}{ \rho_n}=\frac14\Big(\frac{\rho_n-\rho_{n+1}}{\rho_n}\Big)\ge\frac1{2^{n+4}};
\]
when $|y|\ge R$ and $|x|\le \rho$, there holds
\[
\frac{|y-x|}{|y|}\ge1-\frac{\rho}{ R}\ge\frac12, 
\]
provided $\rho\le\frac12R$.
Consequently, using these observations and the fact that $u\ge 0$ a.e. in $\mathcal{Q}$, we estimate the second term as
\begin{align}\nonumber
&\iint_{Q_n}  \z^pw_{-}(x,t)\,\dx\dt \bigg[\essup_{\substack{x\in\widehat{K}_n}} \int_{\rn\setminus K_n}\frac{ w_{-}^{p-1}(y,t)}{|x-y|^{N+sp}}\,\dy\bigg]\\ \nonumber
&=\iint_{Q_n}  \z^pw_{-}(x,t)\,\dx\dt \essup_{\substack{x\in\widehat{K}_n}} \bigg[\int_{K_R\setminus K_n}\frac{ w_{-}^{p-1}(y,t)}{|x-y|^{N+sp}}\,\dy+\int_{\rn\setminus K_R}\frac{ w_{-}^{p-1}(y,t)}{|x-y|^{N+sp}}\,\dy\bigg]\\ \nonumber
&\le \boldsymbol\gm 2^{(N+sp)n} \iint_{Q_n}  \z^pw_{-}(x,t)\,\dx\dt \bigg[ \int_{K_R\setminus K_n}\frac{ w_{-}^{p-1}(y,t)}{|y|^{N+sp}}\,\dy\bigg]\\ \nonumber
&\quad+\boldsymbol\gm   \iint_{Q_n}  \z^pw_{-}(x,t)\,\dx\dt \bigg[ \int_{\rn\setminus K_R}\frac{ w_{-}^{p-1}(y,t)}{|y|^{N+sp}}\,\dy\bigg]\\ \nonumber
&\le \boldsymbol\gm 2^{(N+sp)n} \frac{(\xi\boldsymbol\om)^{p-1}}{\rho^{sp}} \iint_{Q_n}  \z^pw_{-}(x,t)\,\dx\dt \\ \nonumber
&\quad+\boldsymbol\gm   \iint_{Q_n}  \z^pw_{-}(x,t)\,\dx\dt \bigg[\frac{(\xi\boldsymbol\om)^{p-1}}{\rho^{sp}}+ \int_{\rn\setminus K_R}\frac{ u_{-}^{p-1}(y,t)}{|y|^{N+sp}}\,\dy\bigg]\\\nonumber
&\le \boldsymbol\gm 2^{(N+sp)n} \frac{(\xi\boldsymbol\om)^{p-1}}{\rho^{sp}} \iint_{Q_n}  \z^pw_{-}(x,t)\,\dx\dt\\ \nonumber
&\quad +
\boldsymbol\gm   \iint_{Q_n}  \z^pw_{-}(x,t)\,\dx\dt\int_{\rn\setminus K_R}\frac{ u_{-}^{p-1}(y,t)}{|y|^{N+sp}}\,\dy.
\end{align}
The last term in the above estimate will cancel with  the third, negative term on the right-hand side of the energy estimate \eqref{Eq:energy-DG},
if we choose $2\widetilde{\boldsymbol\gm}=\boldsymbol\gm$. As a result of this choice, the second and the third terms in \eqref{Eq:energy-DG} together are bounded by
\[
 \boldsymbol\gm 2^{(N+sp)n} \frac{(\xi\boldsymbol\om)^p}{\rho^{sp}} |A_n|.
\]

Collecting these estimates on the right-hand side of \eqref{Eq:energy-DG}, we arrive at
\begin{align*}
\essup_{-\theta\tilde{\rho}^{sp}_n<t<0}&\int_{\widetilde{K}_n} w_-^2\,\dx +
\int_{-\theta\tilde{\rho}^{sp}_n}^{0}\int_{\widetilde{K}_n}\int_{\widetilde{K}_n}  \frac{|w_-(x,t) - w_-(y,t)|^p}{|x-y|^{N+sp}}\,\dx\dy\dt\\
&\le \boldsymbol\gm 2^{(N+2p)n}\frac{(\xi\boldsymbol\om)^p}{\dl\rho^{sp}}|A_n|.
\end{align*}
Departing from here we can run a similar iteration scheme as in \cite[Lemma~3.1]{Liao-cvpd-24}, and conclude that, there exists $\nu$ depending only on the data $\{s, p, N, C_o, C_1\}$ and $\dl$, such that 
if
\[
\big|\big\{u(x,t)+\ell(t)\le \xi\boldsymbol\om\big\}\cap  Q_{\varrho}(\theta)\big|
	\le
	\nu|Q_{\varrho}(\theta)|,
\]
then
\[
u(x,t)+\ell(t)\ge\tfrac12 \xi\boldsymbol\om\quad
	\mbox{a.e.~in $ Q_{\frac{1}2\varrho}(\theta)$.}
\]
This implies that, if
$\ell(T_2)\le  \tfrac14\xi\boldsymbol\om$ and
if
\[
\big|\big\{u \le  \xi\boldsymbol\om\big\}\cap  Q_{\varrho}(\theta)\big|
	\le
	\nu|Q_{\varrho}(\theta)|,
\]
then
\[
u \ge\tfrac14 \xi\boldsymbol\om\quad
	\mbox{a.e.~in $ Q_{\frac{1}2\varrho}(\theta)$.}
\]
The proof is concluded by redefining $4\widetilde{\boldsymbol\gm}$ as $\widetilde{\boldsymbol\gm}$.
\end{proof}

When quantitative information is known at the initial level, we can propagate it without a time-lag.
\begin{lemma}\label{Lm:DG:initial:1}
Let $u$ be a locally bounded, local weak sub(super)-solution to \eqref{Eq:1:1} -- \eqref{Eq:K} in $E_T$, and let $\xi\in(0,1)$. 
There exist $\nu_o\in(0,1)$ and $\widetilde{\boldsymbol\gm}>1$ depending only on the data $\{s, p, N, C_o, C_1\}$ and independent of $\xi$, such that if
\[
\pm\big(\boldsymbol\mu^{\pm}-u(\cdot, t_o)\big)\ge \xi\boldsymbol\om \quad\text{ a.e. in } K_{\rho}(x_o),
\]
then either
\begin{equation*}
\widetilde{\boldsymbol\gm}
{\rm Tail}\big[\big(u - \boldsymbol \mu^{\pm}\big)_\pm; \mathcal{Q}\big]
>\xi\boldsymbol\om,
\end{equation*}
or
\[
\pm\big(\boldsymbol\mu^{\pm}-u\big)\ge \tfrac14\xi\boldsymbol\om\quad\text{ a.e. in }K_{\frac12\rho}(x_o)\times\big(t_o, t_o+\nu_o(\xi\boldsymbol\om)^{2-p}\rho^{sp}\big],
\]
provided the cylinders are included in $\mathcal{Q}$.
\end{lemma}
\begin{proof}
Let us assume $(x_o,t_o)=(0,0)$ and show the case of super-solutions with $\boldsymbol\mu^-=0$. 
Introduce $k_n$,  $\rho_n$, $\tilde{\rho}_n$, $\hat{\rho}_n$, $\bar{\rho}_n$, $K_n$, $\widetilde{K}_n$, $\widehat{K}_n$ and $\overline{K}_n$ as in \cite[Lemma~3.1]{Liao-cvpd-24}. 
In addition, define the cylinders $Q_n=K_n\times(0,\theta\rho^{sp})$, $\widetilde{Q}_n=\widetilde{K}_n\times(0,\theta\rho^{sp})$, $\widehat{Q}_n=\widehat{K}_n\times(0,\theta\rho^{sp})$ and $\overline{Q}_n=\overline{K}_n\times(0,\theta\rho^{sp})$. 
The cutoff function $\z(x)$ in $K_n$ is chosen to vanish outside $\widehat{K}_n$, be equal to $1$ in $\widetilde{K}_n$, and satisfy $|D\z|\le 2^{n+4}/\rho$. The functions $w_-$ and $\ell(t)$ are the same as in Lemma~\ref{Lm:DG:1}.
With these choices, the energy estimate of Proposition~\ref{Prop:2:1} written in $Q_n$ becomes
 \begin{align}\nonumber
	\essup_{0<t<\theta\rho^{sp}}&\int_{\widetilde{K}_n} w^2_{-}(x,t)\,\dx+ \int^{\theta\rho^{sp}}_{0}\int_{\widetilde{K}_n}\int_{\widetilde{K}_n}  \frac{|w_-(x,t) - w_-(y,t)|^p}{|x-y|^{N+sp}}\,\dx\dy\dt\\\nonumber
	&\quad\le
	\boldsymbol\gm\int^{\theta\rho^{sp}}_{0}\int_{K_n}\int_{K_n}\max\big\{w^p_{-}(x,t), w^p_{-}(y,t)\big\} \frac{|\z(x,t) - \z(y,t)|^p}{|x-y|^{N+sp}}\,\dx\dy\dt\\\nonumber
	&\qquad+\boldsymbol\gm\iint_{Q_n} \z^p w_{-}(x,t)\,\dx\dt \bigg(\essup_{\substack{x\in \widehat{K}_n}} \int_{\rn\setminus K_n}\frac{ w_{-}^{p-1}(y,t)}{|x-y|^{N+sp}}\,\dy\bigg)\\\label{Eq:energy-DG-1}
	&\qquad-2\iint_{Q_n}  \ell^{\prime}(t) \z^pw_{-}(x,t)\,\dx\dt.
\end{align}
The right-hand side of \eqref{Eq:energy-DG-1} is treated as in Lemma~\ref{Lm:DG:1}. In this procedure, we select $\widetilde{\boldsymbol\gm}$ in the definition of $\ell(t)$, such that the second and the third terms on the right-hand side can be packed.   Consequently, we obtain that
 \begin{align*}
\essup_{0<t<\theta\rho^{sp} }&\int_{\widetilde{K}_n} w_-^2\,\dx +
\int^{\theta\rho^{sp}}_{0}\int_{\widetilde{K}_n}\int_{\widetilde{K}_n}  \frac{|w_-(x,t) - w_-(y,t)|^p}{|x-y|^{N+sp}}\,\dx\dy\dt\\
&\le \boldsymbol\gm 2^{(N+2p)n}\frac{(\xi\boldsymbol\om)^p}{\rho^{sp}}|A_n|,
\end{align*}
where  $A_n:=\{u(x,t)+\ell(t)<k_n\}\cap Q_n$. Using this energy estimate one can run the DeGiorgi iteration and obtain a constant $\nu_o\in(0,1)$ depending only on the data, such that
\[
u(x,t)+\ell(t)\ge\tfrac12 \xi\boldsymbol\om\quad
	\mbox{a.e.~in $ K_{\frac12\rho} \times\big(0,\nu_o(\xi\boldsymbol\om)^{2-p}\rho^{sp}\big]$.}
\]
This means that, if we impose $\ell(T_2)\le\tfrac14 \xi\boldsymbol\om$, then
\[
u \ge\tfrac14 \xi\boldsymbol\om\quad
	\mbox{a.e.~in $ K_{\frac12\rho} \times\big(0,\nu_o(\xi\boldsymbol\om)^{2-p}\rho^{sp}\big]$.}
\]
The proof is concluded by redefining $4\widetilde{\boldsymbol\gm}$ as $\widetilde{\boldsymbol\gm}$.
\end{proof}

The following lemma propagates measure theoretical information forward in time.

\begin{lemma}\label{Lm:3:1}
 Let $u$ be a locally bounded, local weak sub(super)-solution to \eqref{Eq:1:1} -- \eqref{Eq:K} in $E_T$.
Introduce parameters $\xi$ and $\al$ in $(0,1)$. There exist  $\dl,\,\varep\in(0,1)$ depending only on the data $\{s, p, N, C_o, C_1\}$ and $\al$, such that if
	\begin{equation*}
	\big|\big\{
		\pm\big(\boldsymbol \mu^{\pm}-u(\cdot, t_o)\big)\ge \xi\boldsymbol \om
		\big\}\cap K_{\varrho}(x_o)\big|
		\ge\al |K_{\varrho}|,
	\end{equation*}
	then 
	either 
$$ 
\frac1\dl 
{\rm Tail}\big[\big(u - \boldsymbol \mu^{\pm}\big)_\pm; \mathcal{Q}\big]
>\xi\boldsymbol\om,	
$$ 
	or
	\begin{equation*}
	\big|\big\{
	\pm\big(\boldsymbol \mu^{\pm}-u(\cdot, t)\big)\ge \varep \xi\boldsymbol \om\big\} \cap K_{\varrho}(x_o)\big|
	\ge\frac{\al}2 |K_\varrho|
	\quad\mbox{ for all $t\in\big(t_o,t_o+\dl(\xi\boldsymbol \om)^{2-p}\varrho^{sp}\big]$,}
\end{equation*}
provided this cylinder is included in $\mathcal{Q}$. Moreover, we have $\varep\approx\al$ and $\dl\approx\al^{p+N+1}$.
\end{lemma}
\begin{proof}
Assuming $(x_o,t_o)=(0,0)$, one shows the case of super-solution with $\boldsymbol \mu^-=0$.
The argument runs exactly like in \cite[Lemma~3.3]{Liao-cvpd-24}. One first writes down the energy estimate for the truncation $w_-=(u-\xi\boldsymbol \om)_-$ in $Q=K_\rho\times(0, \dl(\xi\boldsymbol \om)^{2-p}\varrho^{sp}]$ and with a properly chosen cutoff function $\z(x)$ as in \cite[Lemma~3.3]{Liao-cvpd-24}.
The only difference is the second term on the right-hand side of the energy estimate. In fact, with the same notation and $\z$, we can estimate
\begin{align*}
\iint_{Q} & \z^p w_{-}(x,t)\,\dx\dt \bigg(\essup_{\substack{x\in\supp\z}} \int_{\rn\setminus K_\rho}\frac{ w_{-}^{p-1}(y,t)}{|x-y|^{N+sp}}\,\dy\bigg)\\
&\le \boldsymbol\gm  \frac{\xi\boldsymbol\om |K_\rho|}{\sig^{N+sp}}  \int_0^{\dl(\xi\boldsymbol \om)^{2-p}\varrho^{sp}}\int_{\rn\setminus K_\rho}\frac{ w_{-}^{p-1}(y,t)}{|y|^{N+sp}}\,\dy \dt \\
&= \boldsymbol\gm  \frac{\xi\boldsymbol\om |K_\rho|}{\sig^{N+sp}} \int_0^{\dl(\xi\boldsymbol \om)^{2-p}\varrho^{sp}}\bigg(\int_{K_R\setminus K_\rho}\frac{ w_{-}^{p-1}(y,t)}{|y|^{N+sp}}\,\dy +  \int_{\rn\setminus K_R}\frac{ w_{-}^{p-1}(y,t)}{|y|^{N+sp}}\,\dy\bigg)\dt\\
&\le \boldsymbol\gm  \frac{\xi\boldsymbol\om |K_\rho|}{\sig^{N+sp}} \bigg(\boldsymbol\gm  \dl\xi\boldsymbol\om +\int_0^{\dl(\xi\boldsymbol \om)^{2-p}\varrho^{sp}} \int_{\rn\setminus K_R}\frac{ u_{-}^{p-1}(y,t)}{|y|^{N+sp}}\,\dy\dt\bigg)\\
&\le \boldsymbol\gm  \frac{\dl(\xi\boldsymbol\om)^2 |K_\rho|}{\sig^{N+sp} }.  
\end{align*}
In the last line, we enforced 
\[
\frac1\dl \int^{T_2}_{T_1}\int_{\rn\setminus K_R}\frac{ u_{-}^{p-1}(y,t)}{|y|^{N+sp}}\,\dy\dt\le\xi\boldsymbol\om.
\]
Now, 
we end up with 
the same estimate as in \cite[Lemma~3.3]{Liao-cvpd-24}. Consequently, the selection of parameters $\dl$ and $\varep$ can be done analogously.
 \end{proof}

The measure shrinking lemma can also be readily derived. For simplicity, the vertex $(x_o,t_o)$ is omitted from $Q_\rho(\theta)$.  
\begin{lemma}\label{Lm:3:2}
 Let $u$ be a locally bounded, local weak sub(super)-solution to \eqref{Eq:1:1} -- \eqref{Eq:K} in $E_T$.
 For  some $\dl$, $\sig$ and $\xi$ in $(0,\tfrac12)$, let $\theta=\dl(\sig\xi\boldsymbol \om)^{2-p}$. Suppose that
	\begin{equation*}
	\big|\big\{
		\pm\big(\boldsymbol \mu^{\pm}-u(\cdot, t)\big)\ge \xi\boldsymbol \om
		\big\}\cap K_{\varrho}(x_o)\big|
		\ge\al \big|K_{\varrho}\big|\quad\mbox{ for all $t\in\big(t_o-\theta\varrho^{sp}, t_o\big]$.}
	\end{equation*}
There exists
 $\boldsymbol \gm>1$ depending only on the data $\{s, p, N, C_o, C_1\}$ and independent of $\{\al, \dl, \sig,\xi\}$,  such that 
  either 
  $$ 
\frac1\dl  
{\rm Tail}\big[\big(u - \boldsymbol \mu^{\pm}\big)_\pm; \mathcal{Q}\big]
>\sig\xi\boldsymbol\om,
  $$ 
  or
\begin{equation*}
	\big|\big\{
	\pm\big(\boldsymbol \mu^{\pm}-u\big)\le \sig \xi\boldsymbol \om \big\}\cap Q_{\rho}(\theta)\big|
	\le \boldsymbol\gm \frac{\sig^{p-1}}{\dl\al} |Q_{\rho}(\theta)|,
\end{equation*}
provided $Q_{2\rho}(\theta)$ is included in $\mathcal{Q}$.
\end{lemma}
\begin{proof}
Assuming $(x_o,t_o)=(0,0)$, one shows the case of super-solution with $\boldsymbol \mu^-=0$.
The argument runs exactly like in \cite[Lemma~3.4]{Liao-cvpd-24}. 
One first writes down the energy estimate for the truncation $w_-=(u-\sig\xi\boldsymbol \om)_-$ in $K_{2\rho}\times(-\theta\varrho^{sp},0]$ and with a properly chosen cutoff function $\z(x)$. The only difference is the second term on the right-hand side of the energy estimate. 
With the same notation and $\z$ as in \cite[Lemma~3.4]{Liao-cvpd-24}, we can estimate
\begin{align*}
\int_{-\theta\rho^{sp}}^0\int_{K_{2\rho}} &\z^p w_{-}(x,t)\,\dx\dt \bigg(\essup_{\substack{x\in K_{\frac32\rho} }} \int_{\rn\setminus K_{2\rho}}\frac{ w_{-}^{p-1}(y,t)}{|x-y|^{N+sp}}\,\dy\bigg)\\
&\le \boldsymbol\gm(\sig\xi\boldsymbol \om) |K_{2\rho} | \bigg(\int_{-\theta\rho^{sp}}^0 \int_{\rn\setminus K_{2\rho}}\frac{ w_{-}^{p-1}(y,t)}{|y|^{N+sp}}\,\dy\bigg)\\
&\le \boldsymbol\gm(\sig\xi\boldsymbol \om) |K_{2\rho}| \bigg(\boldsymbol\gm\dl(\sig\xi\boldsymbol \om) +\int_{-\theta\rho^{sp}}^0 \int_{\rn\setminus K_{R}}\frac{ u_{-}^{p-1}(y,t)}{|y|^{N+sp}}\,\dy\dt\bigg)\\
&\le \boldsymbol\gm\dl(\sig\xi\boldsymbol \om)^2 |K_{2\rho}|=2^N\boldsymbol\gm\frac{(\sig\xi\boldsymbol \om)^p}{\rho^{sp}}|Q_{\rho}(\theta)|.
\end{align*}
To obtain the last line, we enforced 
\[
\frac1\dl \int^{T_2}_{T_1}\int_{\rn\setminus K_R}\frac{ u_{-}^{p-1}(y,t)}{|y|^{N+sp}}\,\dy\dt\le\sig\xi\boldsymbol\om.
\]
Now,  
we end up with the same estimate as in \cite[Lemma~3.4]{Liao-cvpd-24}. 
Therefore, we can conclude the proof just like in \cite{Liao-cvpd-24}.
\end{proof}


\section{Proof of Theorem~\ref{Thm:1:1}: $1<p\le 2$}

\subsection{Expansion of Positivity}
Let the cylinder $\mathcal{Q}$ and the numbers $\boldsymbol \mu^{\pm}$ and $\boldsymbol \om$ be defined as in Section~\ref{S:prelim}.   The following  expansion of positivity is in order. The main difference from \cite[Proposition~4.1]{Liao-cvpd-24} lies in the tail alternative.
\begin{proposition}\label{Prop:expansion}
Let $u$ be a locally bounded, local, weak sub(super)-solution to \eqref{Eq:1:1} -- \eqref{Eq:K} in $E_T$, with $1<p\le2$.
Suppose for some constants  $\al,\xi \in(0,1)$, there holds
	\begin{equation*}
		\big|\big\{\pm\big(\boldsymbol \mu^{\pm}-u(\cdot, t_o)\big)\ge \xi\boldsymbol\om \big\}\cap K_{\varrho}(x_o) \big|
		\ge
		\al |K_\varrho |.
	\end{equation*}
There exist a constant $\widetilde{\boldsymbol\gm}>1$ depending only on the data $\{s, p, N, C_o, C_1\}$ and   constants $\dl,\,\eta\in(0,1)$ depending on the data and $\al$, such that 
either
\begin{equation*}
\widetilde{\boldsymbol\gm}
{\rm Tail}\big[\big(u - \boldsymbol \mu^{\pm}\big)_\pm; \mathcal{Q}\big]
>\eta\xi\boldsymbol\om,
\end{equation*}
 or
\begin{equation*}
	\pm\big(\boldsymbol \mu^{\pm}-u\big)\ge\eta\xi\boldsymbol\om
	\quad
	\mbox{a.e.~in $ K_{2\varrho}(x_o) \times\big( t_o+\tfrac12 \dl (\xi\boldsymbol\om)^{2-p}\varrho^{sp},
	t_o+\dl (\xi\boldsymbol\om)^{2-p}\varrho^{sp}\big],$}
\end{equation*}
provided 
\[
K_{4\rho}(x_o)\times\big(t_o, t_o+\dl (\xi\boldsymbol\om)^{2-p}\varrho^{sp}\big]\subset \mathcal{Q}.
\]
Moreover,   $\dl\approx\al^{p+N+1}$ and $\eta\approx\al^q$ for some $q>1$ depending on the data $\{s, p, N, C_o, C_1\}$.
\end{proposition}
\begin{proof}
Assuming $(x_o,t_o)=(0,0)$ and $\boldsymbol\mu^{-}=0$ for simplicity, it suffices to deal with super-solutions.
 Rewriting the measure theoretical information at the initial time $t_o=0$ in the larger ball $K_{4\rho}$ and replacing $\al$ by $4^{-N}\al$, we 
 apply Lemma~\ref{Lm:3:1} to obtain $\dl, \varep\in(0,1)$ depending only on the data $\{s, p, N, C_o, C_1\}$ and $\al$, such that 
	\begin{equation*}
	\big|\big\{
	 u(\cdot, t) \ge \varep \xi\boldsymbol \om\big\} \cap K_{4\varrho} \big|
	\ge\frac{\al}2 4^{-N} |K_{4\varrho}|
	\quad\mbox{ for all $t\in\big(0, \dl(\xi\boldsymbol \om)^{2-p}(4\varrho)^{sp}\big]$,}
\end{equation*}
provided we enforce that
\[
\frac1{\dl}
{\rm Tail} ( u_-; \mathcal{Q})
\le\xi\boldsymbol\om.
\]

This measure theoretical information for each slice of the time interval in turn allows us to apply Lemma~\ref{Lm:3:2} in the cylinders $(0,\bar{t})+Q_{4\rho}(\tfrac12\dl(\sig\varep\xi\boldsymbol \om)^{2-p})$ with an arbitrary 
\begin{equation}\label{Eq:t-bar}
\bar{t}\in \big(\tfrac12\dl(\sig\varep\xi\boldsymbol \om)^{2-p}(4\varrho)^{sp}, \dl(\xi\boldsymbol \om)^{2-p}(4\varrho)^{sp}\big],
\end{equation}
and with $\xi$ and $\al$ there replaced by $\varep\xi$ and $\tfrac12 4^{-N}\al$. This is viable because $\sig\in(0,1)$ and  $\dl(\sig\varep\xi\boldsymbol \om)^{2-p}\le \dl(\xi\boldsymbol \om)^{2-p}$; consequently, we have  
\[
(0,\bar{t})+Q_{4\rho}(\tfrac12\dl(\sig\varep\xi\boldsymbol \om)^{2-p})\subset K_{4\rho}\times\big(0, \dl(\xi\boldsymbol \om)^{2-p}(4\varrho)^{sp}\big]
\]
when $\bar{t}$ ranges over the interval in \eqref{Eq:t-bar}.
Note also this step used the fact that $p\le2$.

Letting $\nu$ be determined in Lemma~\ref{Lm:DG:1}  in terms of the data and $\dl$, we further choose $\sig$ according to Lemma~\ref{Lm:3:2} to satisfy
\[
\boldsymbol\gm \frac{\sig^{p-1}}{\dl\al} <\nu, \quad\text{i.e.}\quad \sig\le\Big(\frac{\nu\dl\al}{\boldsymbol\gm}\Big)^{\frac1{p-1}}.
\]
This choice is possible because $\boldsymbol\gm$ of Lemma~\ref{Lm:3:2} is independent of $\sig$.
Letting $\widetilde{\boldsymbol\gm}$ be chosen in Lemma~\ref{Lm:DG:1} and further enforcing
\begin{equation*}
	\max\Big\{\widetilde{\boldsymbol\gm},\frac1{\dl}\Big\}
	{\rm Tail}(u_-; \mathcal{Q})
	 \le \sig\varep \xi \boldsymbol\om,
\end{equation*}
such a choice of $\sig$ permits us to first apply Lemma~\ref{Lm:3:2} and then Lemma~\ref{Lm:DG:1} in the cylinders $(0,\bar{t})+Q_{4\rho}(\tfrac12\dl(\sig\varep\xi\boldsymbol \om)^{2-p})$ with an arbitrary $\bar{t}$ as in \eqref{Eq:t-bar},
and with $\xi$ there replaced by $\sig\varep\xi$. Therefore, by arbitrariness of $\bar{t}$ we conclude that
\[
u\ge\tfrac14\sig\varep\xi\boldsymbol\om\quad\text{a.e. in}\>K_{2\varrho}  \times\big(\tfrac12 \dl(\sig\varep\xi\boldsymbol\om)^{2-p}(4\rho)^{sp}, \dl (\xi\boldsymbol\om)^{2-p}(4\varrho)^{sp}\big].
\]
The proof is completed by defining $\eta=\tfrac14\sig\varep$ and relabelling $\max\{\widetilde{\boldsymbol\gm},1/\dl\}$ as $\widetilde{\boldsymbol\gm}$. 
\end{proof}
\begin{remark}\upshape
For local operators, a result like Proposition~\ref{Prop:expansion} requires more involved techniques, cf.~\cite[Chap.~IV, Sect.~14]{DB}, \cite[Chap.~4, Sect.~5]{DBGV-mono}.
\end{remark}
Resting upon Proposition~\ref{Prop:expansion}, we are ready to prove Theorem~\ref{Thm:1:1} for $1<p\le 2$. It is noteworthy that all estimates in the proof are stable as $p\to2$. 
\subsection{The First Step} 
Consider the following cylinders at $(x_o,t_o)$:
$$
Q_o:=Q_R(\boldsymbol\om^{2-p})\subset 
Q_{\widetilde{R}}
\subset E_T
$$ 
with positive parameters $\widetilde{R}>R$ and $\boldsymbol\om$ satisfying 
\begin{equation}\label{Eq:intr-cyl}
\boldsymbol\om\ge 2\essup_{Q_{\widetilde{R}}} |u| +{\rm Tail}\big[u; Q_{\widetilde{R}}\big],
\end{equation}
and a smaller cylinder at $(x_o,t_o)$:
$$
\widetilde{Q}_o:=Q(R,\boldsymbol\om^{2-p}(cR)^{sp})
\subset Q_o 
$$ 
for some $c\in(0,\frac14)$ to be chosen.
Set
\begin{equation*}
	\boldsymbol \mu^+:=\essup_{Q_o}u,
	\qquad
	\boldsymbol\mu^-:=\essinf_{Q_o}u.
\end{equation*}
Without loss of generality, we take $(x_o,t_o)=(0,0)$. 
As a result of \eqref{Eq:intr-cyl}, the following intrinsic oscillation estimate holds true:
\begin{equation}\label{Eq:start-cylinder}
 \essosc_{Q_R(\boldsymbol\om^{2-p})}u\le\boldsymbol\om.
\end{equation}
This is the starting estimate of the induction argument to follow.

Let $\dl \in(0,1)$ be determined in Proposition~\ref{Prop:expansion} with $\al=\tfrac12$. 
For some $c\in(0,\tfrac14)$ to be chosen, define $$\tau:=\dl(\tfrac14 \boldsymbol\om)^{2-p}( c R)^{sp}$$ and consider two alternatives
\begin{equation*}
\left\{
\begin{array}{ll}
	\big|\big\{u\big(\cdot,-\tau\big)-\boldsymbol\mu^->\tfrac14 \boldsymbol\om\big\}
	\cap 
	K_{c R}\big|
	\ge
	\tfrac12 |K_{c R}|,\\[5pt]

	\big|\big\{\boldsymbol\mu^+ - u\big(\cdot,-\tau\big)>\tfrac14 \boldsymbol\om\big\}
	\cap 
	K_{c R}\big|
	\ge
	\tfrac12 |K_{c R}|.
\end{array}\right.
\end{equation*}
Assuming $\boldsymbol\mu^+ - \boldsymbol\mu^-\ge\tfrac12\boldsymbol\om$, one of the two alternatives must hold.
Whereas the case $\boldsymbol\mu^+ - \boldsymbol\mu^-<\tfrac12\boldsymbol\om$ will be trivially incorporated into the forthcoming oscillation estimate \eqref{Eq:osc:0}.

Let us suppose the first alternative holds for instance. An appeal to
Proposition~\ref{Prop:expansion} with $\al=\tfrac12$, $\xi=\tfrac14$ and $\rho=cR$ determines $\eta\in(0,\tfrac12)$ and yields that,
either
\begin{equation*}
	\widetilde{\boldsymbol\gm}{\rm Tail}\big[ \big(u-\boldsymbol \mu^{-}\big)_{-}; \widetilde{Q}_o\big] > \eta  \boldsymbol\om,
\end{equation*}
 or
\begin{equation*}
	u-\boldsymbol \mu^{-} \ge\eta\boldsymbol\om
	\quad\text{a.e. in}\> Q_{cR}(\tfrac12\dl(\tfrac14 \boldsymbol\om)^{2-p}).
\end{equation*}
In any case, taking \eqref{Eq:start-cylinder} into account and denoting $\widehat{\boldsymbol\gm}=\widetilde{\boldsymbol\gm}/\eta$, this gives 
\begin{equation}\label{Eq:osc:0}
\essosc_{Q_{cR}(\frac12\dl(\frac14 \boldsymbol\om)^{2-p})} u\le \max\Big\{(1-\eta)\boldsymbol\om,\, \widehat{\boldsymbol\gm}{\rm Tail}\big[ \big(u-\boldsymbol \mu^{-}\big)_{-}; \widetilde{Q}_o\big]\Big\}=:\boldsymbol\om_1.
\end{equation}

At this moment, it is unclear,  due to the presence of the tail, why $\boldsymbol\om_1$ should be controlled by $\boldsymbol\om$. Therefore, tail estimates have to be performed and refined at different stages.
To this end, we first estimate the tail by
\begin{equation}\label{Eq:tail:0}
\begin{aligned}
{\rm Tail}\big[ \big(u-\boldsymbol \mu^{-}\big)_{-}; \widetilde{Q}_o\big] &=  \int^0_{-\boldsymbol\om^{2-p}(cR)^{sp}}\int_{\rn\setminus K_{R}}\frac{\big(u-\boldsymbol \mu^{-}\big)_{-}^{p-1}}{|x|^{N+sp}}\,\dx\dt \\
&\le \boldsymbol\gm c^{sp} \boldsymbol\om +\boldsymbol\gm \int^0_{-\boldsymbol\om^{2-p}(cR)^{sp}}\int_{\rn\setminus K_{R}}\frac{u_{-}^{p-1}}{|x|^{N+sp}}\,\dx\dt\\
&= \boldsymbol\gm c^{sp} \boldsymbol\om +\boldsymbol\gm \int^0_{-\boldsymbol\om^{2-p}(cR)^{sp}}\int_{K_{\widetilde{R}}\setminus K_{R}}\frac{u_{-}^{p-1}}{|x|^{N+sp}}\,\dx\dt \\
&\quad+\boldsymbol\gm \int^0_{-\boldsymbol\om^{2-p}(cR)^{sp}}\int_{\rn\setminus K_{\widetilde{R}}}\frac{u_{-}^{p-1}}{|x|^{N+sp}}\,\dx\dt\\
&\le\boldsymbol\gm c^{sp}\boldsymbol\om+ \boldsymbol\gm \int^0_{-\boldsymbol\om^{2-p}(cR)^{sp}}\int_{\rn\setminus K_{\widetilde{R}}}\frac{u_{-}^{p-1}}{|x|^{N+sp}}\,\dx\dt.
\end{aligned}
\end{equation}
By the intrinsic relation \eqref{Eq:intr-cyl}, the above integral can be bounded by $\boldsymbol\om$, and hence, recalling also the definition of $\boldsymbol\om_1$ in \eqref{Eq:osc:0} we obtain
\begin{equation}\label{Eq:tail:1}
{\rm Tail}\big[ \big(u-\boldsymbol \mu^{-}\big)_{-}; \widetilde{Q}_o\big] \le \boldsymbol\gm \boldsymbol\om,\quad\boldsymbol\om_1\le\overline{\boldsymbol\gm}\boldsymbol\om,
\end{equation}
where $\overline{\boldsymbol\gm}$ and $\boldsymbol\gm$ depend only on the data.

Next, we introduce $R_1=\lm R$ for some $\lm\le c$ to verify the set inclusion
\begin{equation}\label{Eq:lm:0}
Q_{R_1}(\boldsymbol\om_1^{2-p}) \subset Q_{cR}(\tfrac12\dl(\tfrac14 \boldsymbol\om)^{2-p}),\quad\text{i.e.}\quad \lm\le 2^{\frac{2p-5}{p}}\overline{\boldsymbol\gm}^{\frac{p-2}{p}}\dl^{\frac1p}c.
\end{equation}
Here, we employed \eqref{Eq:tail:1}$_2$ to estimate.
As a result of this inclusion and \eqref{Eq:osc:0} we obtain
\begin{equation*}
Q_{R_1}(\boldsymbol\om_1^{2-p})\subset Q_o\quad\text{ and }\quad
 \essosc_{Q_{R_1}(\boldsymbol\om_1^{2-p})}u\le\boldsymbol\om_1,
\end{equation*}
which plays the role of \eqref{Eq:start-cylinder} in the next stage.
At this stage, $c\in(0,\tfrac14)$ is still to be chosen. We also remark that if the second alternative holds instead, one only needs to replace $(u-\boldsymbol \mu^{-})_{-}$ by $(u-\boldsymbol \mu^{+})_{+}$ in \eqref{Eq:osc:0}, and then perform similar calculations to reach the same conclusion.

\subsection{The Induction}
Now we may proceed by induction. 
 Suppose up to $i=1,\cdots, j$, we have built
\begin{equation*}
\left\{
	\begin{array}{c}
	\dsty R_o=R,\quad R_i=\lm R_{i-1},\quad \tfrac12 \boldsymbol\om_{i-1}\le\boldsymbol\om_i\le \overline{\boldsymbol\gm}\boldsymbol\om_{i-1},\\ 
	\dsty\boldsymbol\om_o=\boldsymbol\om,\quad\boldsymbol\om_i=\max\Big\{(1-\eta)\boldsymbol\om_{i-1},\,  \widehat{\boldsymbol\gm} {\rm Tail}\big[ \big(u-\boldsymbol \mu_{i-1}^{\pm}\big)_{\pm}; \widetilde{Q}_{i-1}\big]\Big\},\\[5pt]
	\quad 
	Q_i=Q_{R_i}(\boldsymbol\om_i^{2-p}),\quad \widetilde{Q}_{i}=K_{R_i}\times\big(-\boldsymbol\om_i^{2-p} (cR_i)^{sp},0\big],\quad Q_i\subset Q_{i-1}, 
	\\[5pt]
	\dsty\boldsymbol\mu_i^+=\essup_{Q_i}u,
	\quad
	\boldsymbol\mu_i^-=\essinf_{Q_i}u,
	\quad
	\essosc_{Q_i}u\le\boldsymbol\om_i.
	\end{array}
\right.
\end{equation*}
Note that the relation $\tfrac12 \boldsymbol\om_{i-1}\le\boldsymbol\om_i$ results from the definition of $\boldsymbol\om_i$ and $\eta<\tfrac12$.
The induction argument will show that the above oscillation estimate continues to hold for the $(j+1)$-th step.

Let  $\dl$ be fixed as before, whereas $c\in(0,1)$ is subject to a further choice. To reduce the oscillation in the next stage, we basically repeat what has been done in the first step, now with $\boldsymbol\mu^{\pm}_j$, $\boldsymbol\om_j$, $R_j$, $Q_j$, etc. 
In fact, we define $$\tau:=\dl(\tfrac14 \boldsymbol\om_j)^{2-p}( c R_j)^{sp}$$ and consider two alternatives
\begin{equation}\label{Eq:meas-alter}
\left\{
\begin{array}{ll}
	\big|\big\{u\big(\cdot,-\tau\big)-\boldsymbol\mu^-_j>\tfrac14 \boldsymbol\om_j\big\}
	\cap 
	K_{c R_j}\big|
	\ge
	\tfrac12 |K_{c R_j}|,\\[5pt]

	\big|\big\{\boldsymbol\mu^+_j - u\big(\cdot,-\tau\big)>\tfrac14 \boldsymbol\om_j\big\}
	\cap 
	K_{c R_j}\big|
	\ge
	\tfrac12 |K_{c R_j}|.
\end{array}\right.
\end{equation}
Like in the first step, we may assume $\boldsymbol\mu^+_j - \boldsymbol\mu^-_j\ge\tfrac12\boldsymbol\om_j$, so that one of the two alternatives must hold.
Otherwise, the case $\boldsymbol\mu^+_j - \boldsymbol\mu^-_j<\tfrac12\boldsymbol\om_j$ can be trivially incorporated into the forthcoming oscillation estimate \eqref{Eq:osc:j}.

Let us suppose the first alternative holds for instance. An application of
Proposition~\ref{Prop:expansion} in $Q_j$, with $\al=\tfrac12$, $\xi=\tfrac14$ and $\rho=cR_j$ yields (for the same $\eta$ as before) that,
either
\begin{equation*}
	\widetilde{\boldsymbol\gm} {\rm Tail}\big[ \big(u-\boldsymbol \mu_j^{-}\big)_{-}; \widetilde{Q}_j\big] > \eta  \boldsymbol\om_j,
\end{equation*}
 or
\begin{equation*}
	u-\boldsymbol \mu_j^{-} \ge\eta\boldsymbol\om_j
	\quad\text{a.e. in}\> Q_{cR_j}(\tfrac12\dl(\tfrac14 \boldsymbol\om_j)^{2-p}),
\end{equation*}
which, thanks to the $j$-th induction assumption, gives 
\begin{equation}\label{Eq:osc:j}
\essosc_{Q_{cR_j}(\frac12\dl(\frac14 \boldsymbol\om_j)^{2-p})} u\le \max\Big\{(1-\eta)\boldsymbol\om_j,\, \widehat{\boldsymbol\gm}{\rm Tail}\big[ \big(u-\boldsymbol \mu_j^{-}\big)_{-}; \widetilde{Q}_j\big]\Big\}=:\boldsymbol\om_{j+1}.
\end{equation}
Here, we have taken the same $\widehat{\boldsymbol\gm}$ as in \eqref{Eq:osc:0}.

In order for the induction, it suffices to first obtain $\boldsymbol\om_{j+1}\le \overline{\boldsymbol\gm}\boldsymbol\om_{j}$ for some $\overline{\boldsymbol\gm}>1$. It has the same nature as \eqref{Eq:tail:1}, as long as  $\overline{\boldsymbol\gm}$ can be computed in terms of the data only and independent of $j$. This hinges upon the control of the tail.
Indeed, 
 we rewrite the tail as follows:
 \begin{equation}\label{Eq:tail-rewrite}
\begin{aligned}
{\rm Tail}\big[ \big(u-\boldsymbol \mu_j^{-}\big)_{-}; \widetilde{Q}_j\big]&= \int^0_{-\boldsymbol\om_j^{2-p}(cR_j)^{sp}}\int_{\rn\setminus K_j} \frac{\big(u-\boldsymbol \mu_j^{-}\big)_{-}^{p-1}}{|x|^{N+sp}}\,\dx \dt\\
&=    \int^0_{-\boldsymbol\om_j^{2-p}(cR_j)^{sp}} \int_{\rn\setminus K_R}\frac{\big(u-\boldsymbol \mu_j^{-}\big)_{-}^{p-1}}{|x|^{N+sp}}\,\dx\dt\\
&\quad+ \sum_{i=1}^{j} \int^0_{-\boldsymbol\om_j^{2-p}(cR_j)^{sp}} \int_{K_{i-1}\setminus K_{i}} \frac{\big(u-\boldsymbol \mu_j^{-}\big)_{-}^{p-1}}{|x|^{N+sp}}\,\dx\dt.
\end{aligned}
\end{equation}
Here, we denoted $K_i=K_{R_i}$ for short.
To estimate the first integral on the right of \eqref{Eq:tail-rewrite}, observe that since $|\boldsymbol \mu_j^{-}|\le \boldsymbol \om$ and $u_-\le \boldsymbol \om$ on $Q_{\widetilde{R}}$, we have for any $t\in (-\boldsymbol\om_j^{2-p}R_j^{sp},0)$ that,
\begin{align*}
\int_{\rn\setminus K_R} \frac{\big(u-\boldsymbol \mu_j^{-}\big)_{-}^{p-1}}{|x|^{N+sp}}\,\dx
&\le 
\boldsymbol \gm\int_{\rn\setminus K_R} \frac{|\boldsymbol \mu_j^{-}|^{p-1}+u_{-}^{p-1}}{|x|^{N+sp}}\,\dx\\
&\le\boldsymbol \gm\frac{\boldsymbol \om^{p-1}}{R^{sp}}+\boldsymbol \gm\int_{K_{\widetilde{R}}\setminus K_R} \frac{u_{-}^{p-1}}{|x|^{N+sp}}\,\dx+\boldsymbol \gm\int_{\rn\setminus K_{\widetilde{R}}} \frac{u_{-}^{p-1}}{|x|^{N+sp}}\,\dx\\
&\le\boldsymbol \gm\frac{\boldsymbol \om^{p-1}}{R^{sp}}+\boldsymbol \gm\int_{\rn\setminus K_{\widetilde{R}}} \frac{u_{-}^{p-1}}{|x|^{N+sp}}\,\dx;
\end{align*}
 by integrating the last display in time and using the definition of $\boldsymbol \om$, $\boldsymbol \om_j$ and $R_j$, we continue to estimate
\begin{align*}
& \int^0_{-\boldsymbol\om_j^{2-p}(cR_j)^{sp}}  \int_{\rn\setminus K_R} \frac{\big(u-\boldsymbol \mu_j^{-}\big)_{-}^{p-1}}{|x|^{N+sp}}\,\dx\dt\\
&\qquad\le\boldsymbol \gm  \int^0_{-\boldsymbol\om_j^{2-p}(cR_j)^{sp}}  \bigg(\frac{\boldsymbol \om^{p-1}}{R^{sp}}+ \int_{\rn\setminus K_{\widetilde{R}}} \frac{u_{-}^{p-1}}{|x|^{N+sp}}\,\dx\bigg)\dt\\
&\qquad\le\boldsymbol \gm\boldsymbol\om_j^{2-p}(cR_j)^{sp} \frac{\boldsymbol \om^{p-1}}{R^{sp}} + \boldsymbol \gm  \int^0_{-\boldsymbol\om_j^{2-p}(cR_j)^{sp}}    \int_{\rn\setminus K_{\widetilde{R}}} \frac{u_{-}^{p-1}}{|x|^{N+sp}}\,\dx \dt.
\end{align*}
Whereas the second integral on the right of \eqref{Eq:tail-rewrite} is estimated by using the simple fact that, for $i=1,2,\cdots,j$,
$$\big(u-\boldsymbol \mu_j^{-}\big)_{-}\le \boldsymbol \mu_j^{-} - \boldsymbol \mu_{i-1}^{-}\le\boldsymbol \mu_j^{+} - \boldsymbol \mu_{i-1}^{-}\le\boldsymbol \mu_{i-1}^{+} - \boldsymbol \mu_{i-1}^{-} \le \boldsymbol \om_{i-1}\quad\text{a.e. in}\>Q_{i-1}.$$
 Namely, for any $t\in (-\boldsymbol\om_j^{2-p}(cR_j)^{sp},0)$,
\begin{align*}
\int_{K_{i-1}\setminus K_{i}} \frac{\big(u-\boldsymbol \mu_j^{-}\big)_{-}^{p-1}}{|x|^{N+sp}}\,\dx
\le \boldsymbol \gm \frac{\boldsymbol \om_{i-1}^{p-1}}{R_i^{sp}};
\end{align*} 
consequently, we have
\begin{align*}
& \int^0_{-\boldsymbol\om_j^{2-p}(cR_j)^{sp}}  \int_{K_{i-1}\setminus K_{i}} \frac{\big(u-\boldsymbol \mu_j^{-}\big)_{-}^{p-1}}{|x|^{N+sp}}\,\dx\dt
\le\boldsymbol \gm \boldsymbol\om_j^{2-p}(cR_j)^{sp}\frac{\boldsymbol \om_{i-1}^{p-1}}{R_i^{sp}}.
\end{align*}
Combining the above estimates in \eqref{Eq:tail-rewrite} we arrive at 
\begin{align}\nonumber
{\rm Tail}\big[ \big(u-\boldsymbol \mu_j^{-}\big)_{-}; \widetilde{Q}_j\big] &\le \boldsymbol \gm \boldsymbol\om_j^{2-p}(cR_j)^{sp} \frac{\boldsymbol \om^{p-1}}{R^{sp}} +\boldsymbol \gm\sum_{i=1}^{j}  \boldsymbol\om_j^{2-p}(cR_j)^{sp}\frac{\boldsymbol \om_{i-1}^{p-1}}{R_i^{sp}}\\ \label{Eq:tail-rewrite:1}
&\quad+ \boldsymbol \gm  \int^0_{-\boldsymbol\om_j^{2-p}(cR_j)^{sp}}    \int_{\rn\setminus K_{\widetilde{R}}} \frac{u_{-}^{p-1}}{|x|^{N+sp}}\,\dx \dt.
\end{align}

Next, we continue to estimate the first two terms on the right-hand side of \eqref{Eq:tail-rewrite:1} by observing the simple fact that (see the definitions of  $\boldsymbol \om_j$ and $R_j$)
\begin{equation}\label{Eq:om-omj}
2^{i-j}\boldsymbol \om_i 
\le\boldsymbol \om_j,\quad R_j=\lm^{j-i}R_i,\quad \forall\,i\in\{0,1,\cdots, j\}.
\end{equation}
Using \eqref{Eq:om-omj} the first term on the right-hand side of \eqref{Eq:tail-rewrite:1} is bounded by (noting $\lm\le c$)
$$
\boldsymbol \gm \boldsymbol\om_j^{2-p}(cR_j)^{sp} \frac{\boldsymbol \om^{p-1}}{R^{sp}}\le \boldsymbol \gm \boldsymbol\om_j ( 2^{p-1}c^{sp}  )^j.
$$ 
By the same token, the second term on the right-hand side of \eqref{Eq:tail-rewrite:1} is bounded by 
\begin{align*}
\boldsymbol \gm\sum_{i=1}^{j}  \boldsymbol\om_j^{2-p}(cR_j)^{sp}\frac{\boldsymbol \om_{i-1}^{p-1}}{R_i^{sp}}
\le \boldsymbol \gm \boldsymbol\om_j\sum_{i=1}^{j}(2^{p-1} c^{sp} )^{j-i+1}.
\end{align*}
Substituting these estimates back to \eqref{Eq:tail-rewrite:1} we obtain
\begin{align}\nonumber
{\rm Tail}\big[ \big(u-\boldsymbol \mu_j^{-}\big)_{-}; \widetilde{Q}_j\big] &\le \boldsymbol \gm \boldsymbol\om_j\sum_{i=1}^{j}( 2^{p-1}c^{sp}  )^{j-i+1} \\ \label{Eq:tail-rewrite:2}
&\quad + \boldsymbol \gm  \int^0_{-\boldsymbol\om_j^{2-p}(cR_j)^{sp}}    \int_{\rn\setminus K_{\widetilde{R}}} \frac{u_{-}^{p-1}}{|x|^{N+sp}}\,\dx \dt.
\end{align}
Whereas the integral term in \eqref{Eq:tail-rewrite:2} can be bounded by 
\begin{align}\label{Eq:tail-int}
\boldsymbol \gm  &\int^0_{-\boldsymbol\om_j^{2-p}(cR_j)^{sp}}    \int_{\rn\setminus K_{\widetilde{R}}} \frac{u_{-}^{p-1}}{|x|^{N+sp}}\,\dx \dt\\\nonumber
&\quad\le \boldsymbol \gm  \int^0_{-\boldsymbol\om_j^{2-p}(cR_j)^{sp}}    \int_{\rn\setminus K_{\widetilde{R}}} \frac{|\boldsymbol \mu_{j-1}^{-}|^{p-1}+\big(u-\boldsymbol \mu_{j-1}^{-}\big)_{-}^{p-1}}{|x|^{N+sp}}\,\dx \dt\\ \nonumber
&\quad\le \boldsymbol \gm\boldsymbol\om_j^{2-p}(cR_j)^{sp}\frac{\boldsymbol\om^{p-1}}{R^{sp}}+\boldsymbol \gm  \int^0_{-\boldsymbol\om_j^{2-p}(cR_j)^{sp}}    \int_{\rn\setminus K_{\widetilde{R}}} \frac{\big(u-\boldsymbol \mu_{j-1}^{-}\big)_{-}^{p-1}}{|x|^{N+sp}}\,\dx \dt\\ \nonumber
&\quad\le \boldsymbol \gm\boldsymbol\om_j  ( 2^{p-1}c^{sp}  )^{j}+ \boldsymbol \gm{\rm Tail}\big[ \big(u-\boldsymbol \mu_{j-1}^{-}\big)_{-}; \widetilde{Q}_{j-1}\big]\\ \nonumber
&\quad\le \boldsymbol \gm\boldsymbol\om_j  ( 2^{p-1}c^{sp}  )^{j} + \frac{\boldsymbol \gm}{\widehat{\boldsymbol \gm}}\boldsymbol\om_j . 
\end{align}
To obtain the second-to-last line, we used \eqref{Eq:om-omj} and also the definition of the tail. Whereas the last line follows from the definition of $\boldsymbol\om_j$.
Substituting this estimate back to \eqref{Eq:tail-rewrite:2} we obtain
\[
{\rm Tail}\big[ \big(u-\boldsymbol \mu_j^{-}\big)_{-}; \widetilde{Q}_j\big] \le \boldsymbol \gm \boldsymbol\om_j\sum_{i=1}^{j}( 2^{p-1}c^{sp}  )^{j-i+1} + \frac{\boldsymbol \gm}{\widehat{\boldsymbol \gm}}\boldsymbol\om_j .
\]
The summation on the right is bounded by $1$ if
\begin{equation}\label{Eq:choice-c-1}
2^{p-1}c^{sp} <2^{-1}\quad\text{i.e.}\quad c<2^{-\frac1s}.
\end{equation}
Therefore, we have
\begin{equation*}
{\rm Tail}\big[ \big(u-\boldsymbol \mu_j^{-}\big)_{-}; \widetilde{Q}_j\big] \le\boldsymbol \gm\boldsymbol \om_j,\quad  \boldsymbol \om_{j+1}\le \overline{\boldsymbol\gm}\boldsymbol \om_j.
\end{equation*}
Here, $\overline{\boldsymbol\gm}$ can be different from the one in \eqref{Eq:tail:1}, but we always take the greater one.

Let $R_{j+1}=\lm R_j$ for some $\lm\in(0,1)$ to verify the set inclusion
\begin{equation}\label{Eq:lm:j}
Q_{R_{j+1}}(\boldsymbol\om_{j+1}^{2-p}) \subset Q_{cR_j}(\tfrac12\dl(\tfrac14 \boldsymbol\om_j)^{2-p}),\quad\text{i.e.}\quad \lm\le 2^{\frac{2p-5}{p}}\overline{\boldsymbol\gm}^{\frac{p-2}{p}}\dl^{\frac1p}c.
\end{equation}
As a result of the inclusion \eqref{Eq:lm:j} and \eqref{Eq:osc:j} we obtain
\begin{equation*}
Q_{R_{j+1}}(\boldsymbol\om_{j+1}^{2-p})\subset Q_{R_{j}}(\boldsymbol\om_{j}^{2-p})\quad\text{ and }\quad
 \essosc_{Q_{R_{j+1}}(\boldsymbol\om_{j+1}^{2-p})}u\le\boldsymbol\om_{j+1},
\end{equation*}
which completes the induction argument, in the case of the first alternative in \eqref{Eq:meas-alter}. If the second alternative in \eqref{Eq:meas-alter} holds, one needs to work with $(u-\boldsymbol \mu_j^{+})_{+}$ instead of $(u-\boldsymbol \mu_j^{-})_{-}$ starting from \eqref{Eq:osc:j}. However, the tail estimation can be performed along the same lines, and the same conclusion will be reached. Note also that a further requirement on the smallness of $c$ is imposed in \eqref{Eq:choice-c-1}. Yet, this is not the final choice of $c$, and it is subject to another requirement next.

\subsection{Modulus of Continuity}\label{S:modulus}
By construction of last sections, we have obtained $R_n=\lm^n R $, 
$Q_n=Q_{R_n}(\boldsymbol\om_n^{2-p})$ for $n\in\nn$, such that $Q_n\subset Q_{n-1}$,
and the oscillation estimate
\begin{equation}\label{Eq:osc-est-1}
\essosc_{Q_n}u\le \boldsymbol\om_n=\max\Big\{(1-\eta)\boldsymbol\om_{n-1},\, \widehat{\boldsymbol\gm} {\rm Tail}\big[ \big(u-\boldsymbol \mu_{n-1}^{\pm}\big)_{\pm}; \widetilde{Q}_{n-1}\big]\Big\}.
\end{equation}
The goal of this section is to derive an explicit modulus of continuity encoded in this oscillation estimate. 

Let us take on this task by employing \eqref{Eq:tail-rewrite:2} (with $j=n-1$) to estimate the tail on the right-hand side of \eqref{Eq:osc-est-1}.   Indeed, we may further require the smallness of $c$, such that
\[
\widehat{\boldsymbol \gm}\boldsymbol \gm \boldsymbol\om_{n-1}\sum_{i=1}^{n-1}( 2^{p-1}c^{sp}  )^{n-i}\le\frac12 \boldsymbol\om_{n-1},
\]
which is implied if one imposes that
\begin{equation}\label{Eq:choice-c-2}
c\le (2^{p-1}\widehat{\boldsymbol \gm}\boldsymbol \gm)^{-\frac1{sp}}.
\end{equation}
Consequently, combining this estimate with the definition of $\boldsymbol\om_{n}$ in \eqref{Eq:osc-est-1},  we obtain that
\begin{equation*}
\begin{aligned}
\boldsymbol\om_{n} &\le(1-\eta)\boldsymbol\om_{n-1} +   \boldsymbol \gm  \int^0_{-\boldsymbol\om_{n-1}^{2-p}(cR_{n-1})^{sp}}    \int_{\rn\setminus K_{\widetilde{R}}} \frac{|u|^{p-1}}{|x|^{N+sp}}\,\dx \dt\\
&\le(1-\eta)\boldsymbol\om_{n-1} +   \boldsymbol \gm  \int^0_{-\boldsymbol\om^{2-p}(cR)^{sp}}    \int_{\rn\setminus K_{\widetilde{R}}} \frac{|u|^{p-1}}{|x|^{N+sp}}\,\dx \dt.
\end{aligned}
\end{equation*}
Iterating the above estimate  and joining it with the oscillation estimate \eqref{Eq:osc-est-1} yield that
\begin{equation}\label{Eq:osc-est-2}
\essosc_{Q_n}u \le \boldsymbol\om_{n}\le (1-\eta)^n \boldsymbol\om + \boldsymbol \gm  \int^0_{-\boldsymbol\om^{2-p}R^{sp}}    \int_{\rn\setminus K_{\widetilde{R}}} \frac{|u|^{p-1}}{|x|^{N+sp}}\,\dx \dt.
\end{equation}
From the above display and \eqref{Eq:intr-cyl}, we also readily obtain $\boldsymbol\om_{n}\le \boldsymbol\gm_*\boldsymbol\om$ for $\boldsymbol\gm_*=1+\boldsymbol\gm$.

Next, we observe that the sequence $\{\boldsymbol\om_n^{2-p}R_n^{sp}\}_{n\in\nn}$ decreases to $0$. 
Fix some $r\in(0,R)$.   There must be some $n\in\nn$, such that
\[
\boldsymbol\om_{n+1}^{2-p}R_{n+1}^{sp}\le \boldsymbol\om^{2-p}r^{sp}<\boldsymbol\om_n^{2-p}R_n^{sp}.
\]
The right-hand side inequality together with $\boldsymbol\om_{n}\le \boldsymbol\gm_*\boldsymbol\om$ implies that 
$$
\sig r<R_n
\quad\text{and}\quad 
Q_{\sig r}(\boldsymbol\om^{2-p})\subset Q_n
\quad\text{for}\>\sig=\boldsymbol\gm_*^{(p-2)/sp}.
$$ 
 Whereas the left-hand side inequality implies, iterating $\tfrac12\boldsymbol\om_n\le  \boldsymbol\om_{n+1}$, that  
\[
 \boldsymbol\om^{2-p}r^{sp} \ge \boldsymbol\om_{n+1}^{2-p}R_{n+1}^{sp}\ge (2^{p-2}\lm^{sp})^{n+1}\boldsymbol\om^{2-p} R^{sp},
\]
from which we obtain that
\[
\Big(\frac{r}{R}\Big)^{sp}\ge(2^{p-2}\lm^{sp})^{n+1}\quad\implies\quad (1-\eta)^{n+1}\le\Big(\frac{r}{R}\Big)^{\be}
\]
where 
\[
\be=\frac{sp\ln(1-\eta)}{\ln(2^{p-2}\lm^{sp})}.
\]
Note that the choice of $\lm$ is made out of the smaller one among \eqref{Eq:lm:0} and \eqref{Eq:lm:j}, once $c$ is chosen to be the smaller one of \eqref{Eq:choice-c-1} and \eqref{Eq:choice-c-2}.

Finally, collecting all these estimate in \eqref{Eq:osc-est-2}, we have for all $r\in(0,R)$,
\begin{equation}\label{Eq:osc-final}
\essosc_{Q_{\sig r}(\boldsymbol\om^{2-p})}u \le 2 \boldsymbol\om \Big(\frac{r}{R}\Big)^{\be} + \boldsymbol \gm  \int^0_{-\boldsymbol\om^{2-p}R^{sp}}    \int_{\rn\setminus K_{\widetilde{R}}} \frac{|u|^{p-1}}{|x|^{N+sp}}\,\dx \dt.
\end{equation}
Without loss of generality we may assume that the oscillation estimate \eqref{Eq:osc-final} holds with $R$ replaced by some $\overline{R}\in(r,R)$. Then, we take $\overline{R}=(rR)^{\frac12}$ to obtain
\begin{equation*}
\essosc_{Q_{\sig r}(\boldsymbol\om^{2-p})}u \le 2 \boldsymbol\om \Big(\frac{r}{R}\Big)^{\frac{\be}2}  + \boldsymbol \gm  \int^0_{-\boldsymbol\om^{2-p}(rR)^{sp/2}}    \int_{\rn\setminus K_{\widetilde{R}}} \frac{|u|^{p-1}}{|x|^{N+sp}}\,\dx \dt.
\end{equation*}
The proof is concluded by redefining $\be/2$ to be $\be$.

\section{Proof of Theorem~\ref{Thm:1:1}: $p> 2$}
Consider the cylinders
$$
Q_o:=Q_R(L\theta)\subset 
Q_{\widetilde{R}}
\subset E_T
$$ 
with positive parameters $L$, $\widetilde{R}>R$ and $\theta=(\tfrac14\boldsymbol\om)^{2-p}$ satisfying 
\begin{equation}\label{Eq:intr-cyl:1}
\boldsymbol\om\ge 2\essup_{Q_{\widetilde{R}}} |u| +{\rm Tail}\big[u; Q_{\widetilde{R}}\big],
\end{equation}
and a smaller cylinder
$$
\widetilde{Q}_o:=Q(R,L\theta (cR)^{sp})
\subset Q_o 
$$ 
for some $c\in(0,\frac14)$.
The numbers $c$ and $L$ will be chosen in terms of the data $\{s, p, N, C_o, C_1\}$.
Set 
\begin{equation*}
	\boldsymbol \mu^+:=\essup_{Q_R(L \theta)}u,
	\qquad
	\boldsymbol\mu^-:=\essinf_{Q_R(L \theta)}u.
\end{equation*}
Without loss of generality, we take $(x_o,t_o)=(0,0)$.
Because of \eqref{Eq:intr-cyl:1} the following intrinsic relation holds:
\begin{equation}\label{Eq:start-cylinder:1}
\essosc_{Q_R(L \theta)}u\le\boldsymbol\om.  
\end{equation}
Obviously, it also yields that
\[
\essosc_{Q_R(\boldsymbol\om^{2-p})}u\le\boldsymbol\om. 
\]

The main argument parallels Sections~5.1 -- 5.2 of \cite{Liao-cvpd-24}, unfolding along two alternatives. The underlying intrinsic scaling argument traces back to DiBenedetto's work on the parabolic $p$-Laplacian, cf.~\cite[Chapter~III]{DB}.
\subsection{The First Alternative}\label{S:case-1}
In this section, 
we work with $u$ as a super-solution near its
infimum.
Without loss of generality, we may assume 
\begin{equation}\label{Eq:mu-pm-}
	\boldsymbol\mu^+ -\boldsymbol\mu^- >\tfrac12\boldsymbol\om,
\end{equation}
as the other case $\boldsymbol\mu^+ -\boldsymbol\mu^- \le\frac12\boldsymbol\om$ trivially gives a reduction of oscillation.

Suppose 
for some $\bar{t}\in(-(L-1) \theta (cR)^{sp},0]$, there holds
\begin{equation}\label{Eq:1st-alt-meas}
	\big|\big\{u\le\boldsymbol\mu^-+\tfrac14 \boldsymbol\om\big\}
	\cap 
	(0,\bar{t})+Q_{cR}(\theta)\big|\le \nu|Q_{cR}(\theta)|,
\end{equation}
where $\nu$ is the constant determined in Lemma~\ref{Lm:DG:1} (with $\dl=1$) in terms of the data.
According to Lemma~\ref{Lm:DG:1}  with   $\dl=1$, $\xi=\frac14$ and $\rho=cR$, we have either
\begin{equation}\label{Eq:tail2:0}
	\widetilde{\boldsymbol\gm} {\rm Tail}\big[ \big(u-\boldsymbol \mu^{-}\big)_{-}; \widetilde{Q}_o\big] > \tfrac14   \boldsymbol\om,
\end{equation}
or
\begin{equation}\label{Eq:lower-bd}
	u\ge\boldsymbol \mu^-+\tfrac{1}8\boldsymbol\om
	\quad
	\mbox{a.e.~in $(0,\bar{t})+Q_{\frac12 cR}(\theta)$.}
\end{equation}


Next, we use the pointwise estimate in \eqref{Eq:lower-bd} at $t_*=\bar{t}- \theta(\tfrac12 cR)^{sp}$ and apply Lemma~\ref{Lm:DG:initial:1} with $\rho=\tfrac12cR$ to obtain that, for some free parameter $\xi_o\in(0,\tfrac18)$, either
\begin{equation}\label{Eq:tail2:1}
	\widetilde{\boldsymbol\gm} {\rm Tail}\big[ \big(u-\boldsymbol \mu^{-}\big)_{-}; \widetilde{Q}_o \big] > \xi_o   \boldsymbol\om,
\end{equation}
or
\begin{equation}\label{Eq:lower-bd:1}
 u\ge\boldsymbol \mu^-+\tfrac14\xi_o \boldsymbol\om
	\quad
	\mbox{a.e.~in $K_{\frac14cR}\times\big(t_*, t_*+\nu_o(\xi_o\boldsymbol\om)^{2-p}(\tfrac12cR)^{sp}\big]$.}
\end{equation}
We choose the number $\xi_o$ to fulfill
\begin{equation*}
\nu_o(\xi_o\boldsymbol\om)^{2-p}(\tfrac12cR)^{sp}\ge L  (\tfrac14\boldsymbol\om)^{2-p}(cR)^{sp},\quad\text{i.e.}\quad \xi_o=  \tfrac14\Big(\frac{\nu_o }{2^p L}\Big)^{\frac1{p-2}}.
\end{equation*}
Consequently, the estimate \eqref{Eq:lower-bd:1} holds up to $t=0$ and yields that 
\begin{equation}\label{Eq:osc1:0}
\essosc_{Q_{\frac14cR}(\theta)} u\le \big(1-\tfrac14\xi_o \big)\boldsymbol\om.
\end{equation}
Such an oscillation estimate holds if \eqref{Eq:tail2:0} and \eqref{Eq:tail2:1} do not occur. When one of them occurs, we will incorporate it into the forthcoming oscillation estimate \eqref{Eq:osc1:2-}. Keep in mind that the constant $L$ is yet to be determined in terms of the data $\{s, p, N, C_o, C_1\}$. 

%
%

\subsection{The Second Alternative}
In this section, 
we work with $u$ as a sub-solution near its
supremum.
Suppose \eqref{Eq:1st-alt-meas} does not hold for any $\bar{t}\in\big(-(L-1) \theta (cR)^{sp},0\big]$. Because of \eqref{Eq:mu-pm-}, we rephrase it as
\begin{equation*}
	\big|\big\{\boldsymbol\mu^+-u\ge\tfrac14 \boldsymbol\om\big\}
	\cap 
	(0,\bar{t})+Q_{cR}(\theta)\big|> \nu|Q_{cR}(\theta)|.
\end{equation*}
Based on this, we can find some $t_*\in\big[\bar{t}-\theta(cR)^{sp}, \bar{t}-\tfrac12\nu\theta(cR)^{sp}\big]$ satisfying
\begin{equation*}
	\big|\big\{\boldsymbol\mu^+-u(\cdot, t_*)\ge\tfrac14 \boldsymbol\om\big\}
	\cap 
	 K_{cR} \big|> \tfrac12\nu|K_{cR}|.
\end{equation*}
Indeed, if the above inequality were not to hold for any $t_*$ in the given 
interval, then 
\begin{align*}
	\big|\big\{\boldsymbol\mu^+-u\ge\tfrac14\boldsymbol\om\big\}\cap 
	(0,\bar{t})+Q_{cR}(\theta)\big|
	&=
	\int_{\bar{t}-\theta(cR)^p}^{\bar{t} -\frac12\nu\theta(cR)^{sp}}
	\big|\big\{\boldsymbol\mu^+-u(\cdot, s)\ge \tfrac{1}4\boldsymbol\om\big\}
	\cap K_{cR}\big|\,\ds\\
	&\phantom{=\,}
	+\int^{\bar{t}}_{\bar{t}-\frac12\nu\theta(cR)^{sp}}
	\big|\big\{\boldsymbol\mu^+-u(\cdot, s)\ge \tfrac{1}4\boldsymbol\om\big\}
	\cap K_{cR}\big|\,\ds\\
	&<\tfrac12\nu |K_{cR}|\theta(cR)^{sp}\big(1-\tfrac12\nu\big) +\tfrac12\nu\theta(cR)^{sp}
	|K_{cR} |\\
	&<\nu | Q_{cR}(\theta)|,
\end{align*}
which would yield a contradiction.

Starting from this measure theoretical information, we may apply Lemma~\ref{Lm:3:1} (with $\al=\tfrac12\nu$ and $\rho=cR$) to obtain $\dl$ and $\varep$ depending on the data and $\nu$, such that, for some free parameter $\xi_1\in(0,\frac14)$,
either 
\begin{equation}\label{Eq:tail3:0}
\widetilde{\boldsymbol\gm} {\rm Tail}\big[ \big(u-\boldsymbol \mu^{+}\big)_{+}; \widetilde{Q}_o\big] >\xi_1\boldsymbol \om,
\end{equation} 
 or
\begin{equation}\label{Eq:2nd-alt-meas:1}
	\big|\big\{
	 \boldsymbol \mu^{+}-u(\cdot, t) \ge \varep\xi_1\boldsymbol \om\big\} \cap K_{cR} \big|
	\ge\frac{\al}2 |K_{cR}|
	\>\mbox{ for all $t\in\big(t_*,t_*+\dl(\xi_1\boldsymbol \om)^{2-p}(cR)^{sp}\big]$.}
\end{equation}
The number $\xi_1$ is chosen to satisfy
\[
\dl(\xi_1\boldsymbol \om)^{2-p}(cR)^{sp}\ge\theta(cR)^{sp},\quad\text{i.e.}\quad\xi_1=\tfrac14\dl^{\frac1{p-2}}.
\]
This choice guarantees that  \eqref{Eq:2nd-alt-meas:1} holds at the time level $\bar{t}$.
Consequently,   \eqref{Eq:2nd-alt-meas:1} yields
\begin{equation}\label{Eq:2nd-alt-meas:2}
	\big|\big\{
	 \boldsymbol \mu^{+}-u(\cdot, t) \ge \varep\xi_1\boldsymbol \om\big\} \cap K_{cR} \big|
	\ge\frac{\al}2 |K_{cR}|
	\>\mbox{ for all $t\in\big(-(L-1)\theta (cR)^{sp},0\big]$,}
\end{equation}
thanks to the arbitrariness of $\bar{t}$. 

Given \eqref{Eq:2nd-alt-meas:2}, we plan to employ Lemma~\ref{Lm:3:2} with $\dl=1$, $\xi=\varep\xi_1$ and $\rho=cR$ next. 
In order for that, first fix $\nu$ as in Lemma~\ref{Lm:DG:1} (with $\dl=1$) and then select $\sig\in(0,\tfrac12)$ to verify
\[\boldsymbol\gm \frac{\sig^{p-1}}{\al}\le\nu.\]
 This choice is viable because $\boldsymbol\gm$ of Lemma~\ref{Lm:3:2} is independent of $\sig$.
Then,  $L$ is determined by
\begin{equation}\label{Eq:c2:2}
(L-1)\theta (cR)^{sp}\ge(\sig\varep\xi_1\boldsymbol\om)^{2-p}(cR)^{sp},\quad\text{i.e.}\quad L\ge1+(4\sig\varep\xi_1)^{2-p}.
\end{equation}
Since $\{\sig, \varep, \xi_1\}$ have been fixed in terms of the data, now $L$ is also chosen in terms of the data $\{s, p, N, C_o, C_1\}$.
As a result, the measure theoretical information \eqref{Eq:2nd-alt-meas:2} gives that
	\begin{equation*}
	\big|\big\{
		\boldsymbol \mu^{+}-u(\cdot, t)\ge \varep\xi_1\boldsymbol \om
		\big\}\cap K_{cR}\big|
		\ge\al \big|K_{cR}\big|\quad\mbox{ for all $t\in\big( -(\sig\varep\xi_1\boldsymbol \om)^{2-p}(cR)^{sp}, 0\big]$,}
	\end{equation*}
which allows us to apply Lemma~\ref{Lm:3:2}.
Namely, 
either 
 \begin{equation}\label{Eq:tail3:1}
\widetilde{\boldsymbol\gm} {\rm Tail}\big[ \big(u-\boldsymbol \mu^{+}\big)_{+}; \widetilde{Q}_o\big] >\sig\varep\xi_1\boldsymbol\om
 \end{equation} 
 or
\begin{equation*}
	\big|\big\{
	\boldsymbol \mu^{+}-u \le \sig \varep\xi_1\boldsymbol \om \big\}\cap Q_{cR}(\widetilde\theta)\big|
	\le \nu \big|Q_{cR}(\widetilde\theta)|,\quad\text{where}\>\widetilde\theta=(\sig\varep\xi_1\boldsymbol \om)^{2-p}.
\end{equation*}
By Lemma~\ref{Lm:DG:1} (with $\dl=1$), the last display yields
\[
\boldsymbol \mu^{+}-u \ge \tfrac14\sig \varep\xi_1\boldsymbol \om\quad\text{a.e. in}\>Q_{\frac12cR}(\widetilde\theta),
\]
which in turn gives the reduction of oscillation
\begin{equation}\label{Eq:osc1:1}
\essosc_{Q_{\frac12cR}(\widetilde\theta)} u\le \big(1- \tfrac14\sig \varep\xi_1\big)\boldsymbol \om.
\end{equation}


Combining \eqref{Eq:osc1:0} and \eqref{Eq:osc1:1}, we infer that
\begin{equation}\label{Eq:osc1:2}
\essosc_{Q_{\frac14cR}( \theta)} u\le \big(1- \eta\big)\boldsymbol \om,
\end{equation}
where
\[
\eta=\min\big\{\tfrac14\xi_o, \tfrac14\sig \varep\xi_1\big\}.
\]
The oscillation estimate \eqref{Eq:osc1:2} is achieved, assuming that \eqref{Eq:tail2:0}, \eqref{Eq:tail2:1}, \eqref{Eq:tail3:0} and \eqref{Eq:tail3:1} do not occur. Therefore, taking all cases into account, we arrive at
\begin{equation}\label{Eq:osc1:2-}
\essosc_{Q_{\frac14cR}( \theta)} u\le \max\Big\{(1- \eta)\boldsymbol \om,\,\widehat{\boldsymbol\gm} {\rm Tail}\big[ \big(u-\boldsymbol \mu^{\pm}\big)_{\pm}; \widetilde{Q}_o\big]\Big\}=:\boldsymbol \om_1.
\end{equation}
Here, $\widehat{\boldsymbol\gm}$ is determined by the data only.

Next, we aim to show that $\boldsymbol \om_1\le\overline{\boldsymbol \gm}\boldsymbol \om$ for some $\overline{\boldsymbol \gm}>1$ depending only on the data. To this end, we first estimate the tail like in \eqref{Eq:tail:0}, that is, taking the negative truncation for instance,
\begin{align*}
{\rm Tail}\big[ \big(u-\boldsymbol \mu^{-}\big)_{-}; \widetilde{Q}_o\big] &=  \int^0_{-L\theta(cR)^{sp}}\int_{\rn\setminus K_{R}}\frac{\big(u-\boldsymbol \mu^{-}\big)_{-}^{p-1}}{|x|^{N+sp}}\,\dx\dt \\
&\le \boldsymbol\gm c^{sp} \boldsymbol\om +\boldsymbol\gm \int^0_{-L\theta(cR)^{sp}}\int_{\rn\setminus K_{R}}\frac{u_{-}^{p-1}}{|x|^{N+sp}}\,\dx\dt\\
&= \boldsymbol\gm c^{sp} \boldsymbol\om +\boldsymbol\gm \int^0_{-L\theta(cR)^{sp}}\int_{K_{\widetilde{R}}\setminus K_{R}}\frac{u_{-}^{p-1}}{|x|^{N+sp}}\,\dx\dt \\
&\quad+\boldsymbol\gm \int^0_{-L\theta(cR)^{sp}}\int_{\rn\setminus K_{\widetilde{R}}}\frac{u_{-}^{p-1}}{|x|^{N+sp}}\,\dx\dt\\
&\le\boldsymbol\gm c^{sp}\boldsymbol\om+ \boldsymbol\gm \int^0_{-L\theta(cR)^{sp}}\int_{\rn\setminus K_{\widetilde{R}}}\frac{u_{-}^{p-1}}{|x|^{N+sp}}\,\dx\dt.
\end{align*}
Here, the constant $\boldsymbol\gm$ takes $L$ into account.
By \eqref{Eq:intr-cyl:1}, the above integral can be bounded by $\boldsymbol\om$, and hence, recalling also the definition of $\boldsymbol\om_1$ we obtain
\begin{equation*}
{\rm Tail}\big[ \big(u-\boldsymbol \mu^{\pm}\big)_{\pm}; \widetilde{Q}_o\big] \le \boldsymbol\gm\boldsymbol\om,\quad\boldsymbol\om_1\le\overline{\boldsymbol\gm}\boldsymbol\om,
\end{equation*}
where $\overline{\boldsymbol\gm}$ and $\boldsymbol\gm$ depend only on the data. 

Now, set $\theta_1=(\tfrac14\boldsymbol \om_1)^{2-p}$ and $R_1=\lm R$. To prepare the induction, we need to choose $\lm$ to verify the set inclusion
\begin{equation}\label{Eq:lm:0-}
Q_{R_1}(L\theta_1)\subset Q_{\frac14cR}( \theta),\quad\text{i.e.}\quad \lm\le c L^{-\frac1{sp}}2^{\frac{2-p}{sp}-2}. 
\end{equation}
Here, we used a simple fact that $\boldsymbol \om_1\ge\frac12\boldsymbol \om$ to estimate.
As a result of this inclusion and  \eqref{Eq:osc1:2-}, we obtain
\begin{equation*}
Q_{R_1}(L\theta_1)\subset Q_{R}(L\theta)\quad\text{ and }\quad  \essosc_{Q_{R_1}( L\theta_1)} u\le \boldsymbol \om_1,
\end{equation*}
which takes the place of \eqref{Eq:start-cylinder:1} in the next stage. 
On the other hand, the above oscillation estimate easily yields
\[
\essosc_{Q_{R_1}( \boldsymbol \om_1^{2-p})} u\le \boldsymbol \om_1.
\]
Note that the above oscillation estimate also takes into account the reverse case of \eqref{Eq:mu-pm-}. Note also that in the smallness requirement \eqref{Eq:lm:0-} of $\lm$, the constant $L$ has been determined in \eqref{Eq:c2:2} in terms of the data, whereas $c$ is still to be chosen. 

\subsection{The Induction}
Now we may proceed by induction. 
 Suppose up to $i=1,\cdots, j$, we have built
\begin{equation*}
\left\{
	\begin{array}{c}
	\dsty R_o=R,\quad R_i=\lm R_{i-1},\quad	\theta_i=(\tfrac14\boldsymbol\om_i)^{2-p},\quad\tfrac12\boldsymbol \om_{i-1}\le\boldsymbol \om_{i}\le \overline{\boldsymbol \gm}\boldsymbol \om_{i-1},\\[5pt] 
	\boldsymbol\om_o=\boldsymbol\om,\quad\boldsymbol\om_i=\max\Big\{(1- \eta)\boldsymbol \om_{i-1},\,\widehat{\boldsymbol\gm} {\rm Tail}\big[ \big(u-\boldsymbol \mu_{i-1}^{\pm}\big)_{\pm}; \widetilde{Q}_{i-1}\big]\Big\},\\[5pt]
	Q_i=Q_{R_i}(L\theta_i),\quad\widetilde{Q}_i=K_{R_i}\times\big(-L\theta_i(cR_i)^{sp}, 0\big],\quad Q_i\subset Q_{i-1}, 
	\\[5pt]
	\dsty\boldsymbol\mu_i^+=\essup_{Q_i}u,
	\quad
	\boldsymbol\mu_i^-=\essinf_{Q_i}u,
	\quad
	\essosc_{Q_{R_i}(\boldsymbol\om_i^{2-p})}u\le\boldsymbol\om_i.
	\end{array}
\right.
\end{equation*}
The induction argument will show that the above oscillation estimate continues to hold for the $(j+1)$-th step.

In fact, we can repeat all the previous arguments, which now are adapted with  $\boldsymbol\mu^{\pm}_j$, $\boldsymbol\om_j$, $R_j$, $\theta_j$, $Q_j$, etc. In the end, we have a reduction of oscillation parallel with \eqref{Eq:osc1:2}, that is,
\begin{equation*}
\essosc_{Q_{\frac14cR_j}( \theta_j)} u\le \big(1- \eta\big)\boldsymbol \om_j,
\end{equation*}
provided
\[
\widehat{\boldsymbol\gm}{\rm Tail}\big[ \big(u-\boldsymbol \mu_j^{\pm}\big)_{\pm}; \widetilde{Q}_j\big] \le \boldsymbol\om_j.
\]
In any case, we have
\begin{equation}\label{Eq:osc1:3}
\essosc_{Q_{\frac14cR_j}( \theta_j)} u\le \max\Big\{(1- \eta)\boldsymbol \om_{j},\,\widehat{\boldsymbol\gm} {\rm Tail}\big[ \big(u-\boldsymbol \mu_{j}^{\pm}\big)_{\pm}; \widetilde{Q}_{j}\big]\Big\}:=\boldsymbol \om_{j+1}.
\end{equation}

Now, setting $\theta_{j+1}=(\tfrac14\boldsymbol \om_{j+1})^{2-p}$ and $R_{j+1}=\lm R_j$, like in \eqref{Eq:lm:0-} it is straightforward to verify the set inclusion
\begin{equation}\label{Eq:lm:j-}
Q_{R_{j+1}}(L\theta_{j+1})\subset Q_{\frac14cR_j}( \theta_j), \quad\text{if}\quad \lm\le c L^{-\frac1{sp}}2^{\frac{2-p}{sp}-2},
\end{equation}
which, by \eqref{Eq:osc1:3}, implies
\[
Q_{R_{j+1}}(L \theta_{j+1})\subset Q_{R_{j}}(L \theta_{j})\quad\text{and}\quad\essosc_{Q_{R_{j+1}}(L \theta_{j+1})} u \le \boldsymbol \om_{j+1},
\]
and then obviously,
\[
\essosc_{Q_{R_{j+1}}(\boldsymbol\om_{j+1}^{2-p})}u\le\boldsymbol\om_{j+1}.
\]
Therefore,  the induction argument is  completed. 
The previous deduction works for any $c\in(0,\tfrac14)$, and the final choice of $c$ will be made in the next section.
Note also that, the estimates $\boldsymbol \om_{i}\le \overline{\boldsymbol \gm}\boldsymbol \om_{i-1}$ did not play a role in obtaining the last display. However, it will be used next, and hence  we postpone its proof.

\subsection{Modulus of Continuity}
By construction of last sections, we have fixed $L$ in terms of the data and obtained $R_n=\lm^n R $, $\theta_n=(\tfrac14\boldsymbol\om_n)^{2-p}$,
$Q_n=Q_{R_n}(L\theta_n)$ for $n\in\nn$, such that $Q_n\subset Q_{n-1}$,
and the oscillation estimate
\begin{equation}\label{Eq:osc-est-1-}
\essosc_{Q_{R_n}(\boldsymbol\om_n^{2-p})}u\le \boldsymbol\om_n=\max\Big\{(1-\eta)\boldsymbol\om_{n-1},\,  \widehat{\boldsymbol\gm} {\rm Tail}\big[ \big(u-\boldsymbol \mu_{n-1}^{\pm}\big)_{\pm}; \widetilde{Q}_{n-1}\big]\Big\}.
\end{equation}
Deriving an explicit modulus of continuity encoded in this oscillation estimate is similar to Section~\ref{S:modulus}.  To avoid repetition, we only briefly discuss the procedure, highlighting the main differences. 

The tail estimate can be performed just like in \eqref{Eq:tail-rewrite} and \eqref{Eq:tail-rewrite:1}; note that this procedure does not distinguish $p>2$ and $p<2$. Taking the negative truncation for instance, one only needs to replace $\boldsymbol\om_{j}^{2-p}$ there by $L(\tfrac14\boldsymbol\om_j)^{2-p}$, and eventually one reaches an analogous estimate as \eqref{Eq:tail-rewrite:2}. Writing such an estimate with the index $n-1$, we arrive at
\begin{align}\nonumber
{\rm Tail}\big[ \big(u-\boldsymbol \mu_{n-1}^{-}\big)_{-}; \widetilde{Q}_{n-1}\big] &\le \boldsymbol \gm \boldsymbol\om_{n-1}\sum_{i=1}^{n-1}( 2^{p-1}c^{sp}  )^{n-i} \\ \label{Eq:tail-rewrite:2-}
&\quad + \boldsymbol \gm  \int^0_{-L\theta_{n-1}(cR_{n-1})^{sp}}    \int_{\rn\setminus K_{\widetilde{R}}} \frac{u_{-}^{p-1}}{|x|^{N+sp}}\,\dx \dt.
\end{align}
Here, the constant $\boldsymbol\gm$ takes $L$ into account.

Departing from this estimate, we first show that its right-hand side can be bounded by $ \boldsymbol \gm \boldsymbol\om_{n-1}$, if $c$ is properly chosen. In fact, the integral term can be estimated like in \eqref{Eq:tail-int}, that is,
\[
\boldsymbol \gm  \int^0_{-L\theta_{n-1}(cR_{n-1})^{sp}}    \int_{\rn\setminus K_{\widetilde{R}}} \frac{u_{-}^{p-1}}{|x|^{N+sp}}\,\dx \dt
\le \boldsymbol \gm\boldsymbol\om_{n-1}  ( 2^{p-1}c^{sp}  )^{n-1} + \frac{\boldsymbol \gm}{\widehat{\boldsymbol \gm}}\boldsymbol\om_{n-1} . 
\]
Substituting this estimate back in \eqref{Eq:tail-rewrite:2-}, we then choose $c$ to satisfy
\[
 \boldsymbol \gm \boldsymbol\om_{n-1}\sum_{i=1}^{n-1}( 2^{p-1}c^{sp}  )^{n-i}\le\frac12 \boldsymbol\om_{n-1},
\]
which is implied if one imposes that
\begin{equation}\label{Eq:choice-c-2-}
c\le (2^{p+1}\boldsymbol \gm )^{\frac1{sp} }.
\end{equation}
Hence, the right-hand side of \eqref{Eq:tail-rewrite:2-} is bounded by $ \boldsymbol \gm \boldsymbol\om_{n-1}$, and by definition of $\boldsymbol\om_{n}$ in \eqref{Eq:osc-est-1-}, 
\begin{equation}\label{Eq:om-om-n}
\boldsymbol\om_{n}\le\overline{\boldsymbol\gm}\boldsymbol\om_{n-1}.
\end{equation}
Again, $\overline{\boldsymbol\gm}$ only depends on the data.

Next, we refine the above estimates. Indeed, notice that by induction, $\theta_{n-1}R_{n-1}^{sp}\le \theta R^{sp}$, and hence
$$
L\theta_{n-1}(cR_{n-1})^{sp}\le L c^{sp}  \theta R^{sp}=Lc^{sp}4^{p-2}\boldsymbol\om^{2-p}R^{sp}\le \boldsymbol\om^{2-p}R^{sp},
$$
provided we impose that
\begin{equation}\label{Eq:choice-c-3-}
c\le4^{\frac{2-p}{sp}}L^{-\frac{1}{sp}}.
\end{equation}
Under these requirements of $c$, we can further estimate the tail in \eqref{Eq:tail-rewrite:2-}, and then use it to estimate $\boldsymbol\om_n$ defined in \eqref{Eq:osc-est-1-}. Consequently, we see that
\[
\boldsymbol\om_n\le (1-\eta)\boldsymbol\om_{n-1}+\boldsymbol \gm  \int^0_{- \boldsymbol\om^{2-p} R^{sp}}    \int_{\rn\setminus K_{\widetilde{R}}} \frac{|u|^{p-1}}{|x|^{N+sp}}\,\dx \dt.
\]
Iterating the above estimate  and joining it with the oscillation estimate \eqref{Eq:osc-est-1-} yield that
\begin{equation}\label{Eq:osc-est-2-}
\essosc_{Q_{R_n}(\boldsymbol\om_n^{2-p})}u \le (1-\eta)^n \boldsymbol\om + \boldsymbol \gm  \int^0_{- \boldsymbol\om^{2-p}R^{sp}}    \int_{\rn\setminus K_{\widetilde{R}}} \frac{|u|^{p-1}}{|x|^{N+sp}}\,\dx \dt.
\end{equation}

This is an analog of \eqref{Eq:osc-est-2}, starting from which the final argument runs similarly as in Section~\ref{S:modulus}.
In fact, we observe that the sequence $\{\boldsymbol\om_n^{2-p}R_n^{sp}\}_{n\in\nn}$ decreases to $0$. Moreover, iterating  $\boldsymbol\om_n\ge\tfrac12 \boldsymbol\om_{n-1}$,  we can estimate
\[
\boldsymbol\om_n^{2-p}R_n^{sp}\le (2^{-n}\boldsymbol\om)^{2-p}(\lm^n R)^{sp}=(2^{p-2}\lm^{sp})^{n}\boldsymbol\om^{2-p}R^{sp}\le 2^{-n}\boldsymbol\om^{2-p}R^{sp},
\]
provided we require that 
\begin{equation}\label{Eq:lm:n-}
\lm<2^{-\frac{p-1}{sp}} .
\end{equation}
Fix some $r\in(0,R)$.   There must be some $n\in\nn_0$ such that
\[
\boldsymbol\om_{n+1}^{2-p}R_{n+1}^{sp}\le \boldsymbol\om^{2-p}r^{sp}<\boldsymbol\om_n^{2-p}R_n^{sp}.
\]
The right-hand side inequality implies that 
$$
r<R_n \quad\text{and}\quad Q_r(\boldsymbol\om^{2-p})\subset Q_{R_n}(\boldsymbol\om_n^{2-p}).
$$ 
Whereas the left-hand side inequality implies, iterating \eqref{Eq:om-om-n}, that  
\[
 \boldsymbol\om^{2-p}r^{sp} \ge \boldsymbol\om_{n+1}^{2-p}R_{n+1}^{sp}\ge (\overline{\boldsymbol\gm}^{2-p}\lm^{sp})^{n+1}\boldsymbol\om^{2-p} R^{sp},
\]
from which we obtain that
\[
\Big(\frac{r}{R}\Big)^{sp}\ge(\overline{\boldsymbol\gm}^{2-p}\lm^{sp})^{n+1}\quad\implies\quad (1-\eta)^{n+1}\le\Big(\frac{r}{R}\Big)^{\be}
\]
where 
\[
\be=\frac{sp\ln(1-\eta)}{\ln(\overline{\boldsymbol\gm}^{2-p}\lm^{sp})}.
\]
Note that the choice of $\lm$ is made out of the smaller one among \eqref{Eq:lm:0-}, \eqref{Eq:lm:j-} and \eqref{Eq:lm:n-}, once $c$ is chosen to be the smaller one of \eqref{Eq:choice-c-2-} and \eqref{Eq:choice-c-3-}.

Finally, collecting all these estimate in \eqref{Eq:osc-est-2-}, we have for all $r\in(0,R)$,
\begin{equation*}
\essosc_{Q_r(\boldsymbol\om^{2-p})}u \le 2 \boldsymbol\om \Big(\frac{r}{R}\Big)^{\be} + \boldsymbol \gm  \int^0_{-\boldsymbol\om^{2-p}R^{sp}}    \int_{\rn\setminus K_{\widetilde{R}}} \frac{|u|^{p-1}}{|x|^{N+sp}}\,\dx \dt.
\end{equation*}
An interpolation will conclude the proof in the same way as in Section~\ref{S:modulus}.
\section{Proof of Theorem~\ref{Thm:1:2}}
Let $Q_{\widetilde{R}}\subset E_T$,
and let $(x_o,t_o)=(0,0)$ for simplicity. Define
\begin{equation*}
\begin{aligned}
 \boldsymbol\om:=2\essup_{Q_{\widetilde{R}}} |u| 
+
\bigg(\bint^{0}_{- \widetilde{R}^{sp}}  \Big( \widetilde{R}^{sp} \int_{\rn\setminus K_{\widetilde{R}}} \frac{|u|^{p-1}}{|x|^{N+sp}}\,\dx\Big)^{ 1+\varep} \dt\bigg)^{\frac1{1+\varep}}.
\end{aligned}
\end{equation*}
It is apparent that such $ \boldsymbol\om$ satisfies \eqref{Eq:intr-cyl}, and also $Q_R( \boldsymbol\om^{2-p})\subset Q_{\widetilde{R}}$ by assumption.
Therefore, applying Theorem~\ref{Thm:1:1} we have
	\begin{equation*}
	\essosc_{ Q_{\sig r}( \boldsymbol\om^{2-p})}u 
	\le 2  \boldsymbol\om \Big(\frac{r}{R}\Big)^{\be}  + \boldsymbol \gm  \int^{0}_{- \boldsymbol\om^{2-p}(rR)^{sp/2}}    \int_{\rn\setminus K_{\widetilde{R}}} \frac{|u|^{p-1}}{|x|^{N+sp}}\,\dx \dt
	\end{equation*}
for any $0<r<R$.
The integral term is estimated by H\"older's inequality:
\begin{align*}
&\int^{0}_{-\boldsymbol\om^{2-p}(rR)^{sp/2}}    \int_{\rn\setminus K_{\widetilde{R}}} \frac{|u|^{p-1}}{|x|^{N+sp}}\,\dx \dt\\
&\quad\le  [\boldsymbol\om^{2-p}(rR)^{sp/2}]^{\frac{\varep}{1+\varep}}
\bigg(\int^{0}_{- \boldsymbol\om^{2-p}(rR)^{sp/2}}  \Big(\int_{\rn\setminus K_{\widetilde{R}}} \frac{|u|^{p-1}}{|x|^{N+sp}}\,\dx\Big)^{ 1+\varep} \dt\bigg)^{\frac1{1+\varep}}\\
&\quad\le  \frac{[\boldsymbol\om^{2-p}(rR)^{sp/2}]^{\frac{\varep}{1+\varep}}}{[\boldsymbol\om^{2-p}R^{sp}]^{\frac{\varep}{1+\varep}}}
\bigg(\bint^{0}_{- \widetilde{R}^{sp}}  \Big( \widetilde{R}^{sp} \int_{\rn\setminus K_{\widetilde{R}}} \frac{|u|^{p-1}}{|x|^{N+sp}}\,\dx\Big)^{ 1+\varep} \dt\bigg)^{\frac1{1+\varep}}\\
&\quad\le  \boldsymbol\om \Big(\frac{r}{R}\Big)^{\frac{\varep sp}{2(1+\varep)}}.
\end{align*}
As a result, the desired H\"older estimate follows.

\end{document}